\documentclass[11p]{ouparticle}

\usepackage{bbm}

\usepackage{amsmath}
\usepackage{amsthm}
\usepackage{amsfonts}
\usepackage{mathrsfs}
\usepackage[numbers]{natbib}

\usepackage{xspace}
\newcommand{\psh}[2]{\ensuremath{\langle #1,#2\rangle}\xspace}
\DeclareMathAlphabet{\pazocal}{OMS}{zplm}{m}{n}

\usepackage{color}
\usepackage{graphicx}
\usepackage[T1]{fontenc}
\usepackage{url}
\usepackage{array}
\usepackage{graphicx}
\usepackage{lscape}
\usepackage{mathrsfs}
\usepackage{listings}
\usepackage{etex}
\usepackage{longtable}
\usepackage{siunitx}
\usepackage{mathrsfs}
\usepackage{enumitem}
\usepackage{dsfont}

\DeclareMathOperator{\E}{\mathbb{E}}

\DeclareMathOperator{\Tr}{\text{Tr}}
\DeclareMathOperator*{\argmin}{arg\,min}

\theoremstyle{plain}
\newtheorem{thrm}{Theorem}[section]
\newtheorem{lemma}[thrm]{Lemma}

\newtheorem{proposition}[thrm]{Proposition}

\theoremstyle{definition}
\newtheorem{definition}[thrm]{Definition}

\newtheorem{rem}[thrm]{Remark}

\newtheorem{assumption}[thrm]{Assumption}
\theoremstyle{plain}

\numberwithin{thrm}{section}
\numberwithin{equation}{section}
\begin{document}
\title{Well-posedness of stochastic port-Hamiltonian systems on infinite-dimensional spaces \footnote{This research was conducted with the financial support of F.R.S-FNRS. Fran\c{c}ois Lamoline is a FRIA Research Fellow under the grant F 3/5/5-MCF/BC. This paper presents research results of the Belgian Programme on Interuniversity Poles of Attraction, initiated by the Belgian State, the Prime Minister's Office for Science, Technology and Culture. The scientific responsibility rests with its authors.}}
%
\author{ \name{Fran\c{c}ois Lamoline and Joseph J. Winkin
\footnote{Fran\c{c}ois Lamoline and Joseph J. Winkin are with the University of Namur, Department of Mathematics and Namur Institute for Complex Systems (naXys), Rempart de la Vierge 8, B-5000 Namur, Belgium; email:  francois.lamoline@unamur.be \& joseph.winkin@unamur.be}}} 
         
%
%
%

%
%
\date{}
\abstract{
Stochastic port-Hamiltonian systems on infinite-dimensional spaces governed by It\^o stochastic differential equations (SDEs) are introduced and some properties of this new class of systems are studied. They are an extension of stochastic port-Hamiltonian systems defined on a finite-dimensional state space. The concept of well-posedness in the sense of Weiss and Salamon is generalized to the stochastic context. Under this extended definition, stochastic port-Hamiltonian systems are shown to be well-posed. The theory is illustrated on an example of a vibrating string subject to a Hilbert space-valued Gaussian white noise process.
}
%
\maketitle

\section*{Introduction}
Distributed parameter linear port-controlled Hamiltonian systems were introduced in \cite{gorrec,vander}. Since then, they have been proved to be an efficient framework for modeling and control of partial differential equations (PDEs), such as transmission lines, flexible beams and tubular reactors for instance. This class of systems covers a large range of control systems with actuators and sensors located at the extremities of the spatial domain. Boundary control and observation lead to unbounded operators and thus to technical difficulties such as the definition of well-defined solution and output. However, so far, many boundary control systems have been proved to be well-posed even without having bounded control and observation operators. For an overview of the literature on the well-posedness of boundary control systems, one may be referred to \cite{staffan,Weiss2001, tucsnak}. The well-posedness of the wave equation is treated in \cite{zhang, zhang2}.

From \citep{wellposed}, it is known that, for deterministic port-Hamiltonian systems on a one-dimensional spatial domain, as for the bounded case, the well-posedness can be deduced from the generation of a $C_0$-semigroup in the homogeneous case $(u=0)$. In other words, the port-Hamiltonian modeling allows to have the simplest verification of the PDEs' well-posedness, which is usually very hard to obtain for more general PDEs.     

Various disturbances such as modeling inaccuracies or environment disturbances can occur when real plants are to be controlled. Examples of disturbances are wind gusts, environment turbulences, unpredictable fluctuation in the line voltage, fluctuations of the environment temperature or reaction parameters uncertainty. In order to capture the nature of these neglected effects, stochastic port-Hamiltonian systems (SPHSs) are introduced as the stochastic extension of linear port-controlled Hamiltonian systems by considering system's noise. From a denominational point a view, port-Hamiltonian systems not subject to any disturbance will be called deterministic port-Hamiltonian systems and those subject to disturbances will be called stochastic port-Hamiltonian systems. On finite-dimensional spaces, stochastic port-Hamiltonian systems were defined in \citep{fujimoto} as the stochastic extension of deterministic port-Hamiltonian systems defined in \citep{maschke}. In \citep{fujimoto}, Satoh and Fujimoto depicted the performance degradation and the possible non-stabilization of control systems resulting from stochastic disturbances by considering the problem of controlling a rolling coin on a horizontal plane.\\    
To the best of our knowledge, no stochastic extension of port-Hamiltonian systems on infinite-dimensional spaces with boundary control and observation is already available in the literature.

Linear well-posed systems in the sense of Salamon \cite{salomon} were introduced to deal with systems with boundary control and observation operators. This class of systems is also known to enjoy many useful properties (see e.g. \cite{staffan}) involving feedback control, dynamic stabilization, and tracking/disturbance rejection. It is the main motivation to generalize the well-posedness to stochastic partial differential equations and more specifically, to extend well-posedness results of deterministic port-Hamiltonian systems to stochastic ones. To do so, the semigroup approach will be preferred to the variational one. As a good starting point, interested readers may be referred to \cite{pritchard, curtain, daprato} and the references therein. A brief appendix resuming results from the theory of stochastic integration of deterministic functions in Hilbert spaces relevant to the paper is included for the reader's convenience, see Appendix \ref{appendixA}.\\
As far as known, there are not as many references devoted to stochastic well-posed systems as for the deterministic case, see \cite{Weiss2001,Tuscnakauto,zhang,zhang2,staffan, wellposed}. One may be referred to \citep{Lu15} where a generalization of well-posed linear systems to the stochastic context is done by providing a formulation of stochastic well-posed linear systems and some basic properties. This paper falls in line with \citep{Lu15} as an attempt to fill the blank left in the literature. When compared with \citep{Lu15}, here the system's noise is assumed to be additive and modelized as an infinite-dimensional Wiener process. Furthermore, we focus on the class of PDEs that falls within the port-Hamiltonian framework and thus, use the existing results of the deterministic case \citep{wellposed}.

Built on \citep{Lamolinecdc}, where preliminary results are reported, the first main contribution of this paper is to provide the stochastic counterpart of the port-controlled Hamiltonian systems defined in \citep{gorrec} with additive system's noise and to describe them as boundary controlled and observed stochastic systems. The existence and the uniqueness of weak and strong solutions of SPHSs are studied and some of their properties are investigated. The second main contribution is the extension of the results of \citep{wellposed} to the stochastic context, while \citep{Lamolinecdc} focuses on the passivity property of SPHSs.

This paper is organized as follows. In Section \ref{SPHS_section} the class of systems under study is introduced, namely stochastic port-Hamiltonian systems on infinite-dimensional spaces. In Section \ref{weak} the existence and the uniqueness of strong and weak solutions are established. In Section \ref{stochasticwellposed_section} results concerning the well-posedness of deterministic port-Hamiltonian systems are briefly recalled and the extension of these results to the stochastic case is presented. The paper ends with the study case of a vibrating string subject to a space and time Gaussian white noise process by means of the Riesz-spectral property of a subclass of stochastic port-Hamiltonian systems. 
\section{Stochastic port-Hamiltonian system}
\label{SPHS_section}
Let us consider the Hilbert space $Z$ and the probability space $(\Omega, \mathcal{F},\mathbb{P})$ with a complete right-continuous filtration $\mathbb{F}:=(\mathcal{F}_t)_{t\geq 0}$ and a spatial-dependent stochastic process $\varepsilon(\omega,t)\in L^2([a,b];\mathbb{K}^n)$ with $\omega\in \Omega$ and $t\in[0,T]$ ($\mathbb{K}$ denotes the field of real or complex numbers). Hereinafter, the state space $L^2([a,b];\mathbb{K}^n)$ will be denoted by $\mathcal{X}$.
\begin{definition}
A first order stochastic port-Hamiltonian system is governed by the stochastic partial differential equation (SPDE)
\begin{equation}
\dfrac{\partial \varepsilon}{\partial t}(\zeta ,t) = P_1 \dfrac{\partial}{\partial \zeta}(\mathcal{H}(\zeta)\varepsilon(\zeta ,t)) + P_0 \mathcal{H}(\zeta)\varepsilon(\zeta ,t) + (H \eta(t))(\zeta),
\label{spde}
\end{equation}
where $P_1 \in \mathbb{K}^{n\times n}$ is invertible and self-adjoint ($P_1^{*} = P_1$), $P_0 \in \mathbb{K}^{n\times n}$ is skew-adjoint ($P_0^{*} = -P_0$)  and $\mathcal{H}\in L^{\infty}([a,b];\mathbb{K}^{n\times n})$ is self-adjoint and satisfies $mI\leq \mathcal{H}(\zeta)\leq MI$ for all $\zeta\in [a,b]$, for some constants $m$, $M>0$. The system's noise $\eta:\Omega \times [0,T] \rightarrow Z$ is a Gaussian white noise process and $H\in\pazocal{L}(Z,\mathcal{X})$ is the intensity of $\eta$.\\
The associated Hamiltonian $E:\mathcal{X}\rightarrow \mathbb{K}$, which describes the total energy of the system, is given by 
\begin{equation}
E(\varepsilon(t)) = \frac{1}{2}\int_a^b \varepsilon^*(\zeta,t) \mathcal{H}(\zeta) \varepsilon(\zeta,t) d\zeta,
\label{Hamiltonian}
\end{equation} 
where $\varepsilon(t)$ denotes the function $\zeta \to \varepsilon(\zeta,t)$. The Hamiltonian is assumed to be sufficiently smooth. 
\label{pshspde}
\end{definition}
The expected stored energy will be given by  
\begin{equation}
\bar{E}(\varepsilon(t)):= \frac{1}{2}\E \int_a^b \varepsilon^*(\zeta,t) \mathcal{H}(\zeta) \varepsilon(\zeta,t) d\zeta = \frac{1}{2} \int_a^b \Tr\left[ \mathcal{H}(\zeta) \E\left[\varepsilon(\zeta,t) \varepsilon(\zeta,t)^* \right] \right] d\zeta,
\end{equation}
using the calculation $\E\left[\varepsilon(\zeta,t)^* \mathcal{H}(\zeta) \varepsilon(\zeta,t) \right] = \E\left[\Tr\left[\mathcal{H}(\zeta) \varepsilon(\zeta,t) \varepsilon(\zeta,t)^* \right] \right] = \Tr\left[\mathcal{H}(\zeta) \right.$ $\left.\E\left[\varepsilon(\zeta,t) \varepsilon^*(\zeta,t) \right]\right]$
since the trace of a product is invariant under cyclic permutations.\\
The Hilbert space $L^2([a,b];\mathbb{K}^n)$ is equipped with the inner product
\begin{equation}
\psh{\varepsilon_1}{\varepsilon_2}_{\mathcal{X}} = \frac{1}{2} \int_a^b \varepsilon_2^*(\zeta) \mathcal{H}(\zeta) \varepsilon_1(\zeta)d\zeta,\qquad \varepsilon_1,\varepsilon_2\in L^2([a,b];\mathbb{K}^n),
\end{equation}
which induces the norm $\|\cdot\|_\mathcal{X}=\sqrt{\psh{\cdot}{\cdot}}_{\mathcal{X}}$. We make this choice of norm in order to have $\|\cdot\|_\mathcal{X}$ representing the energy of the vector signal. Since $mI\leq \mathcal{H}(\zeta)\leq MI$ for all $\zeta\in [a,b]$, the norm $\|\cdot\|_\mathcal{X}$ is equivalent to the usual $L^2$-norm.\\
Let us introduce the following spaces:
\begin{equation*}
\begin{split}
L^2_{\mathbb{F}}([0,T];L^2(\Omega;\mathcal{X})):= 
\lbrace \varepsilon:\Omega \times [0,T] \rightarrow \mathcal{X} : \varepsilon(\cdot) \text{  is  } \mathbb{F}-\text{adapted and} \E \int_0^T \| \varepsilon(s)\| _\mathcal{X}^2 ds <\infty  \rbrace
\end{split}
\end{equation*}
endowed with the norm $\|\varepsilon\|^2_{L_{\mathbb{F}}^2([0,T]\times \Omega;\mathcal{X})}:=\E\int_0^T \| \varepsilon(s)\| _\mathcal{X}^2 ds$,
\begin{equation*}
\begin{split}
C^2_{\mathbb{F}}([0,T];L^2(\Omega;\mathcal{X})):= \lbrace \varepsilon: \Omega \times [0,T] \rightarrow \mathcal{X} : \varepsilon(\cdot) \text{  is  } \mathbb{F}-\text{adapted and} \E \| \varepsilon(s)\| _\mathcal{X}^2, \E\|\dot{\varepsilon}(s)\|_\mathcal{X}^2 \text{  are continuous}  \rbrace
\end{split}
\end{equation*}
endowed with the norm $\|\varepsilon\|^2_{C_{\mathbb{F}}([0,T];L^2(\Omega;\mathcal{X}))}:=\displaystyle \sup_{t\in [0,T]} \E \| \varepsilon(t)\|^2_\mathcal{X}$ and 
\begin{equation*}
\begin{split}
M_\mathbb{F}^2([0,T];Z):= \lbrace M:[0,T]\rightarrow Z : M(\cdot) \text{	is a continuous	} \mathbb{F}-\text{adapted martingale}, M(0)=0 \text{	and	}\\ \displaystyle \sup_{t\in [0,T]}\E \|M(t)\|^2_Z <\infty \rbrace
\end{split}
\end{equation*}
endowed with norm $\|M\|^2_{M_T^2}:= \displaystyle \sup_{t\in [0,T]} \E \|M(t)\|^2_Z$, which is a Banach space. In order to simplify the notations we denote the first two spaces by $L^2_\mathbb{F}([0,T];\mathcal{X})$ and $C^2_\mathbb{F}([0,T];\mathcal{X})$ respectively.\\

To the SPDE (\ref{spde}), we associate controlled and homogeneous boundary conditions given by 
\begin{align*}
u(t) = W_{B,1} \left[
\begin{array}{c}
f_\partial(t) \\
e_\partial(t)
\end{array}
\right],\qquad
0 = W_{B,2} \left[
\begin{array}{c}
f_\partial(t) \\
e_\partial(t)
\end{array}
\right], 
\end{align*}
where 
\begin{align}
e_{\partial} &= \frac{1}{\sqrt{2}}((\mathcal{H}\varepsilon)(b) + (\mathcal{H}\varepsilon)(a)),\label{effort}\\
f_{\partial} &= \frac{1}{\sqrt{2}} (P_1 (\mathcal{H}\varepsilon)(b) - P_1 (\mathcal{H}\varepsilon)(a))\label{flow}
\end{align}
are said to be the boundary effort and the boundary flow, respectively, and $W_B:= \begin{scriptsize}
\left[
\begin{array}{c}
W_{B,1} \\
W_{B,2}
\end{array}
\right]
\end{scriptsize} \in \mathbb{K}^{n\times 2n}$. The input $u(t)$ is a $\mathbb{K}^m$-valued $\mathbb{F}$-adapted stochastic process.\\
One will usually model Gaussian white noise disturbances by It\^o stochastic integrals with respect to a Wiener process, see \cite{curtain_article} and \cite{pritchard}. Let us define the operator \begin{equation}
A\varepsilon:= P_1 \dfrac{d}{d\zeta}(\mathcal{H} \varepsilon) + P_0 \mathcal{H} \varepsilon
\label{operator}
\end{equation} on the domain 
\begin{equation}
D(A)= \left\lbrace \varepsilon\in \mathcal{X} : \mathcal{H}\varepsilon\in H^{1}([a,b];\mathbb{K}^n) , W_B 
\left[
 \begin{array}{c}
 f_{\partial}\\
 e_{\partial}
\end{array}
\right]
=0 \right\rbrace.
\label{domain}
\end{equation}
The SPDE (\ref{spde}) can then be rewritten as a stochastic differential equation (SDE) under the It\^o form on the functional state space $\mathcal{X}$:
\begin{equation}
d\varepsilon(t) = A \varepsilon(t) dt + H dw(t),
\label{ito} 
\end{equation}
where $w(t)$ is a Wiener process on a separable Hilbert space $Z$ with covariance operator $Q\in\pazocal{L}(Z)$ and intensity $H\in\pazocal{L}(Z,\mathcal{X})$, see Definition \ref{wienerprocessdefi}. Here the operator $Q\in\pazocal{L}(Z)$ is assumed to be symmetric, nonnegative and to satisfy $\Tr[Q]<\infty$, where $\Tr$ denotes the trace operator of $Q$.\\ 
The following result, establishing the generation of a $C_0$-semigroup for deterministic PHSs, will be useful for our study of the well-posedness of stochastic port-Hamiltonian systems. 
\begin{thrm}
\citep[Theorem 7.2.4]{zwart}\\
Consider the operator $A$ with domain $D(A)$ given by (\ref{operator})-(\ref{domain}). Assume that $W_B$ is a $n\times 2n$ matrix of full rank. Then $A$ is the generator of a contraction $C_0$-semigroup on $\mathcal{X}$ if and only if $W_B \Sigma W_B^*\geq 0$ where 
$\Sigma=\begin{scriptsize} \left[
 \begin{array}{cc}
 0 & I\\
I & 0
\end{array}
\right] \end{scriptsize} \in \mathbb{K}^{2n\times 2n}$. 
\label{generation}
\end{thrm}
Let us introduce the class of boundary controlled and observed (BCO for short) stochastic systems with $m$ inputs and $p$ outputs. We consider the stochastic boundary control system described by control equations:
\begin{equation}
\begin{aligned}
d\varepsilon(t) &= \mathcal{A}\varepsilon(t) dt+ Hdw(t), \qquad \varepsilon(0)=\varepsilon_0,\\
u(t)&=\mathcal{B}\varepsilon(t),\\
y(t) &=\mathcal{C}\varepsilon(t),
\end{aligned}
\label{bcostochastic}
\end{equation}
where $\mathcal{A}:D(\mathcal{A})\rightarrow \mathcal{X}$, $\mathcal{B}:D(\mathcal{B})\rightarrow \mathbb{K}^m$ and $\mathcal{C}:D(\mathcal{C})\rightarrow \mathbb{K}^p$ are unbounded linear operators s.t. $D(\mathcal{A})\subset D(\mathcal{B})\subset\mathcal{X}$. The input $u(t)$ is assumed to be a stochastic process in $L^2_\mathbb{F}([0,T];\mathbb{K}^m)$ and $w(t)$ is a Wiener process on a separable Hilbert space $Z$ with covariance operator $Q$ of trace class and intensity $H\in\pazocal{L}(Z,\mathcal{X})$. The initial condition $\varepsilon_0$ is a $\mathcal{X}$-valued Gaussian random variable with mean $m_{\varepsilon_0}$ and covariance operator $Q_0$. $w$ and $\varepsilon_0$ are assumed to be mutually independent.  
\begin{definition}
A BCO stochastic system is a system described by (\ref{bcostochastic}) which satisfies the following conditions:
\begin{enumerate}
\item[1.] The operator $A : D(A)\rightarrow \mathcal{X}$ defined for every $x\in D(A)= D(\mathcal{A})\cap Ker(\mathcal{B})$ by $Ax=\mathcal{A}x$, is the infinitesimal generator of a $C_0$-semigroup $(T(t))_{t\geq 0}$ on $\mathcal{X}$,
\item[2.] There exists an operator $B\in \pazocal{L}(\mathbb{K}^m,\mathcal{X})$ such that, for every $u\in \mathbb{K}^m$, we have $Bu\in D(\mathcal{A})$, $\mathcal{A}B\in\pazocal{L}(\mathbb{K}^m,\mathcal{X})$ and $\mathcal{B}Bu=u$ for all $u\in \mathbb{K}^m$;
\item[3.] The observation operator $\mathcal{C}\in\pazocal{L}(D(A),\mathbb{K}^p)$, where $D(A)$ is endowed with the graph norm of $A$.
\item[4.] $w(t)$ is a Wiener process with $\Tr[Q]<\infty$ and $H\in L^0_2$, i.e. $\| H\|^2_{L_2^0}:=\Tr[HQH^*] <\infty$ , which ensures that the It\^o integrals $\int_0^t H dw(s)$ and $\int_0^t T(t-s) H dw(s)$ are well-defined. 
\end{enumerate}
\label{def_bcostochastic}
\end{definition}
Note that the space of H-S operators $L_2^0=L_2(Q^{1/2}(Z),\mathcal{X})$ is detailed in Appendix \ref{appendixA}.\\

We are now in position to detail the specific class of SPHSs which will be studied in details in this paper.
\begin{definition}
Boundary controlled and observed stochastic port-Hamiltonian systems are described by 
\begin{align}
\label{phito}
d\varepsilon(t) &= \mathcal{A}\varepsilon(t) dt + Hdw(t), \qquad \varepsilon(0)=\varepsilon_0,\\
\label{inputsto}
u(t) &= W_{B,1} \left[
\begin{array}{c}
f_\partial(t) \\
e_\partial(t)
\end{array}
\right] =: \mathcal{B}\left[ \varepsilon(t)\right],\\
\label{homogeneoussto}
0 &= W_{B,2} \left[
\begin{array}{c}
f_\partial(t) \\
e_\partial(t)
\end{array}
\right],\\ 
\label{outputsto}
y(t) &= W_C \left[
\begin{array}{c}
f_\partial(t) \\
e_\partial(t)
\end{array}
\right]=: \mathcal{C}\left[ \varepsilon(t)\right],
\end{align}
where $W_B:= \begin{scriptsize}
\left[
\begin{array}{c}
W_{B,1} \\
W_{B,2}
\end{array}
\right]
\end{scriptsize} \in \mathbb{K}^{n\times 2n}$ and $W_C\in \mathbb{K}^{p\times 2n}$, $\mathcal{A}$ is a linear operator given by 
\begin{equation}
\mathcal{A}\varepsilon:= P_1 \dfrac{d}{d\zeta}(\mathcal{H} \varepsilon) + P_0 (\mathcal{H} \varepsilon)
\label{A_bc}
\end{equation}
and $\mathcal{B}:D(\mathcal{B})\to \mathbb{K}^m$ is a linear operator, with the same domain  
\begin{equation}
D(\mathcal{A}):= \left\lbrace \varepsilon(t)\in \mathcal{X}: \mathcal{H}\varepsilon(t)\in H^1([a,b];\mathbb{K}^n) \text{	and		} W_{B,2}\left[
 \begin{array}{c}
 f_\partial \\
 e_\partial 
\end{array}
\right]=0 \right\rbrace = D(\mathcal{B}).
\end{equation}
\label{sph}
\end{definition}
The following two conditions will be assumed to hold throughout.
\begin{assumption}
The matrices $W_B$ and $W_C$ are full rank, $W_B$ satisfies $W_B \Sigma W_B^*\geq 0$ and \begin{small} \text{rank	}
$\left[
\begin{array}{c}
W_{B,1} \\
W_C
\end{array}
\right] = m + p$
\end{small}. \label{assumptions3}
\end{assumption}
Notice that hereafter the expression \textit{boundary controlled and observed stochastic port-Hamiltonian systems} will be shortened in stochastic port-Hamiltonian systems (SPHSs). The boundary control and observation will not be specified anymore. From \citep[Theorem 11.3.2]{zwart}, it is known that the SPHS (\ref{phito})-(\ref{homogeneoussto}) is a boundary controlled stochastic system as defined in Definition \ref{def_bcostochastic} and thus, the change of variables for the state: $X(\zeta,t) = \varepsilon(\zeta,t) - Bu(t)$ applied to (\ref{phito}) leads to an associated SDE given by 
\begin{equation}
dX(t) = AX(t) dt - B\dot{u}(t) dt + \mathcal{A}Bu(t) dt + Hdw(t), \qquad X(0)=X_0.
\label{stochasticassociated}
\end{equation} 
\begin{definition}
A Hilbert space-valued process $(X(t))_{t\in[0,T]}$ is said to be a mild solution of (\ref{stochasticassociated}) with respect to $(w(t))_{t\in[0,T]}$ if 
\begin{enumerate}
\item $X(t)$ is $\mathbb{F}$-adapted;
\item $X(t)\in C([0,T];L^2(\Omega;\mathcal{X}))$
\item for all $t\in[0,T]$,  $\mathbb{P}(\omega\in \Omega: \int_0^T \|X(\omega,t)\|^2_\mathcal{X} ds < \infty) =1$  and
\begin{equation}
\begin{split}
X(t)= T(t)X_0 + \int_0^t T(t-s) (\mathcal{A}Bu(s) - B\dot{u}(s)) ds + \int_0^t T(t-s) H dw(s).
\end{split}
\label{solution_formula}
\end{equation}
\end{enumerate}
\end{definition}
Observe that from Condition 4 of Definition \ref{def_bcostochastic}, the stochastic convolution process $W_A(t):= \int_0^t T(t-s)Hdw(s)$ is well-defined. Unlike the stochastic integral, the convolutional stochastic integral is no longer a martingale and is only mean-square continuous. The relation between the mild solutions of (\ref{phito}) and (\ref{stochasticassociated}) is given by $X(\zeta,t) = \varepsilon(\zeta,t) - Bu(t)$.\\
Now, we can state the first specific result of this paper.
\begin{thrm}
Consider a stochastic port-Hamiltonian system (\ref{phito})-(\ref{homogeneoussto}) as in Definition \ref{sph}, satisfying Assumption \ref{assumptions3} and Condition 4 of Definition \ref{def_bcostochastic}. In this setting, the mild solution of (\ref{phito}) is represented as state trajectories of a stochastic process given by (\ref{solution_formula}) and satisfies the following estimate: for any $t>0$, there is a constant $K(t)>0$ such that
\begin{equation}
\E\|X(t)\|^2_\mathcal{X} \leq K(t) \left[\E\| X_0\|^2_\mathcal{X} + \E \|u\|^2_{H^1([0,t];\mathbb{K}^m)} + \Tr[Q] \right],
\end{equation}
where $\|\cdot\|_{H^1([0,t];\mathbb{K}^m)}= \int_0^{t} \|\cdot\|_{\mathbb{K}^m} ds + \int_0^t \|\frac{d(\cdot)}{ds}\|_{\mathbb{K}^m} ds$.
Moreover, if $u$ is deterministic such that $u(t)= \mathcal{B}\E[X(t)]$, for every $t >0$: 
\begin{enumerate}
\item The mean of $X(t)$ is governed by the abstract differential equation (ADE) 
\begin{equation}
\dot{m}_X(t)= T(t) m_{X_0} + \mathcal{A}Bu(t)-B\dot{u}(t),
\label{mean_eq}
\end{equation}
whose mild solution is $m_X(t) = T(t)m_{X_0} + \int_0^t T(t-s) (\mathcal{A}Bu(s)-B\dot{u}(s))ds$. 
\item The variance of $X(t)$ is governed by the Lyapunov type ADE 
\begin{equation}
\dot{Cov}(X(t)) = A Cov(X(t)) + Cov(X(t))A^* + HQH^*,
\label{var_eq}
\end{equation}
whose mild solution is $Cov(X(t)) = T(t)Q_0T(t)^* + \int_0^t T(t-s)HQH^* T(t-s)^* ds$.  
\end{enumerate}
\label{mildsolutionthm}
\end{thrm}
\begin{proof}
The existence and uniqueness of a mild solution can be directly deduced by using a probabilistic fixed point argument and its expression is obtained from the variational constant formula (\ref{solution_formula}). The estimate is obtained by using It\^o's isometry and the boundedness of the operators $\mathcal{A}B$, $B$ and $H\in L_2^0$.  
\begin{align*}
\E\|X(t)\|^2_\mathcal{X} &= \E\| T(t)X_0 + \int_0^t T(t-s) (\mathcal{A}Bu(s) - B\dot{u}(s)) ds + \int_0^t T(t-s) H dw(s) \|^2_\mathcal{X}\\
&\leq 3 \E\|X_0\|^2_\mathcal{X} + 3\E \int_0^t \|\mathcal{A}Bu(s) -B\dot{u}(s) \|^2_\mathcal{X} ds + 3 t\|H\|^2_{L_2^0}\\
&\leq K(t) \left[\E \|X_0\|^2_\mathcal{X} + \E \|u\|^2_{H^1([0,t];\mathbb{K}^m)}+ \Tr[Q] \right]
\end{align*}
$X(t)$ given by (\ref{solution_formula}) is $\mathbb{F}$-adapted since $W_A(t)$ and $u(t)$ are $\mathbb{F}$-adapted and $X_0$ is $\mathcal{F}_0$-measurable. The mean-square continuity of $X(t)$ is a straightforward consequence of the mean-square continuity of $W_A(t)$.\\
Using the vanishing property of the stochastic integral and the fact that $X_0$ has mean $m_{X_0}$, (\ref{mean_eq}) is obtained.\\
Using the independence of $X_0$ and $w(t)$, (\ref{var_eq}) is deduced by Leibniz' differentiation rule. 
\end{proof}
\begin{rem}
Theorem \ref{mildsolutionthm} also holds for general BCO stochastic systems as defined in Definition \ref{def_bcostochastic}.
\end{rem}
In most cases, the mild solutions are not continuous regarding their sample paths $X(\omega,t)$. Indeed, even though the deterministic part of (\ref{solution_formula}) is continuous, the mild solution is not continuous since the stochastic convolution term only satisfies the mean-square continuity. However, since we are considering a specific class of systems, which are stochastic port-Hamiltonian systems, we shall prove in the following result that the continuity of the sample paths holds for this class. The proof of this result is based on the Hausenblas-Seidler approach, see \citep{Hausenblas2001}.\\

\begin{thrm}
Assume that the stochastic port-Hamiltonian system (\ref{phito})-(\ref{homogeneoussto}) satisfying Assumption \ref{assumptions3} and Condition 4 of Definition \ref{def_bcostochastic} admits a mild solution $X(t)$ given by (\ref{solution_formula}). Then, $X(t)$ has continuous sample paths.
\label{continuity}
\end{thrm}
\begin{proof}
Since we already know that the Bochner integral $\int_0^t T(t-s) (\mathcal{A}Bu(s) - B\dot{u}(s)) ds$ and $T(t)X_0$ are continuous, it remains to prove that $\int_0^t T(t-s) H dw(s)$ is continuous $\mathbb{P}$-a.s. Since $(T(t))_{t\geq 0}$ is a contraction $C_0$-semigroup, we can apply the Sz-Nagy-Foias theory of dilations \citep{unitary_dilations} as done in \citep{Hausenblas2001}. Therefore, the $C_0$-semigroup $(T(t))_{t\geq 0}$ has a unitary dilation $(\bar{T}(t))_{t\geq 0}$ on a larger Hilbert space $\mathcal{X}_1$ so that the state space $\mathcal{X}$ is embedded as a closed subspace of $\mathcal{X}_1$. Besides, $(\bar{T}(t))_{t\geq 0}$ is a strongly continuous unitary group on $\mathcal{X}_1$ with $T(t)=P \bar{T}(t)$ for all $t\geq 0$, where $P$ is the orthogonal projection of $\mathcal{X}_1$ onto $\mathcal{X}$. We denote the infinitesimal generator of $(\bar{T}(t))_{t\geq 0}$ as $\bar{A}$. Hence the stochastic convolution of the operator $\bar{A}$ can be expressed as
\begin{equation}
\int_0^t \bar{T}(t-s) H dw(s) = \bar{T}(t) \int_0^t \bar{T}(-s) H dw(s).  
\end{equation}
First, we shall prove that $\int_0^t \bar{T}(-s) H dw(s)$ has continuous sample paths and thus so does $\int_0^t \bar{T}(t-s) H dw(s)$. It is known that if $\E\left[\int_0^t \|\bar{T}(-s) H\|^2_{L_2^0} ds \right]<\infty$, then $\int_0^t \bar{T}(-s) H dw(s)$ is continuous. Next, the continuity of the orthogonal projection $P$ entails that the stochastic convolution term $\int_0^t T(t-s) H dw(s)$ has continuous sample paths, which concludes the proof.  
\end{proof}

In view of It\^o's formula (\ref{itoformula}), see Appendix \ref{appendixA}, an energy balance equation can be obtained. However, It\^o's formula cannot be applied directly to mild solutions. In order to solve this problem, existence and uniqueness of weak and strong solutions will be proved for SPHSs in the subsequent section.   

\section{Existence and uniqueness of weak and strong solutions}
\label{weak}
The concept of weak solution is obtained by applying $z\in D(A^*)$ to both parts of the stochastic differential equation (\ref{stochasticassociated}). 
\begin{definition}
A $\mathcal{X}$-valued process $(X(t))_{t\in [0,T]}$ with $T\geq 0$ is said to be a weak solution of (\ref{stochasticassociated}) with respect to the Wiener process $(w(t))_{t\in [0,T]}$ if the trajectories $X(t)$ are $\mathbb{P}$-a.s Bochner integrable and if for all $z\in D(A^*)$ and $t\in[0,T]$
\begin{equation}
\begin{split}
\psh{X(t)}{z}_{\mathcal{X}} = \psh{X_0}{z}_{\mathcal{X}} + \int_0^t \left[ \psh{X(s)}{A^*z}_{\mathcal{X}} + \psh{\mathcal{A}Bu(t)-B\dot{u}(t)}{z}_{\mathcal{X}}\right]ds + \psh{Hw(t)}{z}_{\mathcal{X}}, \text{		} \mathbb{P}-\text{a.s}.
\end{split}
\label{weak_formula}
\end{equation} 
\end{definition}

Since $\mathcal{A}B$ and $B$ are bounded operators, \citep[Section 5.2]{daprato} can be used to show the existence and uniqueness of a weak solution to (\ref{stochasticassociated}).

\begin{thrm}
Consider a BCO stochastic system (\ref{bcostochastic}) as in Definition \ref{def_bcostochastic}. Then, for every input $u\in C_\mathbb{F}^2([0,T];\mathbb{K}^m)$, $\mathcal{H}\varepsilon_0 \in H^1([a,b];\mathbb{K}^n)$ and $u(0) = W_B  \left[
\begin{array}{c}
f_\partial(0) \\
e_\partial(0)
\end{array}
\right]$,
the stochastic differential equation (\ref{stochasticassociated}) admits a unique weak solution given by (\ref{solution_formula}). Since $X(t)$ defined by (\ref{solution_formula}) is almost surely integrable, the mild and weak solutions coincide.
\label{existence_theorem}
\end{thrm}
\begin{proof}
From \citep[Theorem 10.1.8]{zwart}, it is already known that $x(t)$ given by 
$$
x(t)= T(t)x_0 + \int_0^t T(t-s)(\mathcal{A}Bu(s) - B\dot{u}(s)) ds, \qquad t\geq 0
$$
is the unique weak solution of  
\begin{equation}
\dot{x}(t) = Ax(t)+\mathcal{A}Bu(t)-B\dot{u}(t), \qquad x(0)=x_0. 
\end{equation}
Therefore, it is enough to prove that the process  $\int_0^t T(t-s) Hdw(s)$ is a unique weak solution of
\begin{equation}
dX(t) = AX(t) dt + H dw(t), \qquad X(0)=0.
\label{weaksystem}
\end{equation}
 with $t>0$. For this we refer to the proof of \citep[Theorem 5.4]{daprato}
\end{proof}

The strong solution is more restrictive than the weak one since it must take values in $D(A)$. Therefore, the usual way of defining the solution by integrating both parts of the stochastic equation (\ref{stochasticassociated}) can be applied.

\begin{definition}
A $\mathcal{X}$-valued process $(X(t))_{t\in [0,T]}$ with $T\geq 0$ is said to be a strong solution of (\ref{stochasticassociated}) with respect to the Wiener process $(w(t))_{t\in [0,T]}$ with the covariance operator $Q$ satisfying $\Tr Q < \infty$ if $X(t)$ belongs to $D(A)$, $\int_0^\infty \|A X(s)\|ds <\infty$ $\mathbb{P}$-a.s and the process $(X(t))_{t \in [0,T]}$ is given by
\begin{equation}
X(t) = X_0 + \int_0^t \left(AX(s) + \mathcal{A}Bu(s)-B\dot{u}(s)\right) ds + \int_0^t H dw(s), \qquad \mathbb{P}-\text{a.s}.
\label{strongequation} 
\end{equation} 
\end{definition}  

Observe that $\mathcal{A}Bu(s)-B\dot{u}(s)\in D(A)=D(\mathcal{A})\cap\text{Ker}\mathcal{B}$ would be too restrictive on $u(t)$: see Section \ref{illustration} where the input should then be taken as $u(t)=0$. Therefore, SPHSs with control in the dynamic does not have a strong solution. 
This has the inconvenient that It\^o's formula cannot be applied directly to (\ref{stochasticassociated}). Thus, in order to have a strong solution, we shall extend the state space and build a family of approximate systems by using the Yosida approximate and a limiting argument. The extended state space is defined as $\mathcal{X}^e:=\mathbb{K}^m\oplus \mathcal{X}$, where the (extended) state is defined as $X^e(t):=(\begin{array}{cc}u(t) & X(t) \end{array})^T$ and $\tilde{u}(t)=\dot{u}(t)$. Then
\begin{equation}
\begin{aligned}
dX^e(t) = \left( \begin{array}{cc}
0 & 0\\
\mathcal{A}B & A
\end{array}\right) X^e(t) dt + \left( \begin{array}{c}
I\\
-B
\end{array}\right) \tilde{u}(t) dt + \left( \begin{array}{c}
0\\
H
\end{array}\right)  dw(t).
\end{aligned}
\label{extended}
\end{equation}
Let us define $A^e:=\left( \begin{array}{cc}
0 & 0\\
\mathcal{A}B & A
\end{array}\right)$ and $B^e:= \left( \begin{array}{c}
I\\
-B
\end{array}\right)$ with domains $D(A^e)=\mathbb{K}^m\oplus D(A)$ and $D(B^e)= \mathbb{K}^m$ and $H^e=\left[\begin{array}{c}
0\\
H
\end{array} \right]$. From \citep[Theorem 3.3.4]{zwart}, $A^e$ is the infinitesimal generator of a $C_0$-semigroup $T^e(t)= \left(\begin{array}{cc}
I & 0\\
S(t) & T(t)
\end{array} \right)$, where $S(t)u:= \int_0^t T(t-s) \mathcal{A}Bu(s) ds$ for all $u(t) \in\mathbb{K}^m$. Let us define the following approximate control operator for all $\lambda\in\rho(A^e)$,
\begin{equation}
B^e_\lambda: \mathbb{K}^m \rightarrow \mathcal{X}: \tilde{u} \mapsto B^e_\lambda \tilde{u}:= \lambda R(\lambda,A^e)B^e\tilde{u}, 
\end{equation}
where the resolvent operator $R(\lambda,A^e)=(\lambda I- A^e)^{-1}$.
\begin{thrm}
Consider a BCO stochastic system (\ref{bcostochastic}) as in Definition \ref{def_bcostochastic}. In addition, we assume that $HQ^{1/2}(Z)\subset D(A)$ and that $X_0\in D(A)$. If the following condition holds for all $t\geq 0$:
\begin{align}
\int_0^t \| AT(t-s) H\|^2_{L_2^0} ds <\infty;
\end{align}
then for all $\lambda\in\rho(A^e)$, 
\begin{equation}
dX_{\lambda}^e(t)= A^e X_{\lambda}^e(t) dt + B_\lambda^e \tilde{u}(t) dt + H^e dw(t); \qquad X^e(0)= \left(\begin{array}{cc}
u(0) & X_0
\end{array}\right)^T\in D(A^e),
\end{equation}
has a unique strong solution $X^e_\lambda(t)$ with respect to $w(t)$, where $B^e_\lambda\tilde{u}=\lambda R(\lambda,A^e)B^e\tilde{u}$ for all $\tilde{u}\in\mathbb{K}^m$, such that 
\begin{equation}
\displaystyle \sup_{0\leq s\leq t} \E \|X_\lambda^e(s)-X^e(s)\|_{\mathcal{X}^e}^2 \rightarrow 0 \qquad \text{as	} \lambda \to \infty, 
\label{convergence}
\end{equation}
where $X^e(t)$ is the mild solution of (\ref{extended}).
\label{strongsolution}
\end{thrm}
\begin{proof} 
The uniqueness of the strong solution is a direct outcome of the uniqueness of the mild solution.\\
To prove that the mild solution of (\ref{extended}) satisfies the integral equation
\begin{equation}
X_\lambda^e(t) = X^e_0 + \int_0^t (A^eX_\lambda^e(s) + B_\lambda^e\tilde{u}(s)) ds + \int_0^t H^e dw(s),
\label{integralequation}
\end{equation} 
one can use a similar argumentation as in \citep[Theorem 5.35]{pritchard}. The derivation of identity (\ref{integralequation}) is quite standard and is available in Appendix \ref{appendixB} for interested readers.\\ 
Moreover, the continuity of $X^e_\lambda(t)$ can be deduced from the continuity of $\int_0^t H^e dw(s)$ and, since $A^eX^e_\lambda(t)$ is assumed to be integrable, $\int_0^t A^eX^e_\lambda(s)$ ds has continuous sample paths.\\
We know that $\displaystyle \lim_{\lambda \to \infty} \lambda R(\lambda,A^e)z=z$, $z\in \mathcal{X}^e$. Therefore, since $(T(t))_{t\geq 0}$ is a contraction $C_0$-semigroup and by using $\|\lambda R(\lambda,A^e)\|\leq 2$ for $\lambda$ large enough, we have that 
\begin{equation*}
\displaystyle \sup_{0\leq s\leq t}\E \|X_\lambda^e(s)-X^e(s)\|_{\mathcal{X}^e}^2\\ \leq \int_0^t \E \|(I-\lambda R(\lambda,A^e)) B^e \tilde{u}(r) \|_{\mathcal{X}^e}^2 dr.
\end{equation*}
So, $\E \|X^e_\lambda(s)-X^e(s)\|^2_{\mathcal{X}^e} \to 0$ uniformly on $[0,t]$.
\end{proof}
\begin{rem}
The condition $\int_0^t \| AT(t-s) H\|^2_{L_2^0} ds <\infty$ can be replaced by the stronger assumption that $AHQ^{1/2}$ is a Hilbert-Schmidt operator, i.e. $AHQ^{1/2}\in L_2(Z,\mathcal{X})$. Indeed, note that 
\begin{align*}
\int_0^t \|A T(t-s)H\|^2_{L_2^0} ds &=  \int_0^t \| T(s) AH Q^{1/2} \|^2_{L_2} ds\\ 
&\leq \|AH Q^{1/2}\|^2_{L_2} \int_0^t \|T(s)\|^2 ds <\infty.
\end{align*} 
\end{rem}

The It\^o's formula will now be applied to determine the energy increments due to the noise process, which entails that the passivity property is not preserved for stochastic and linear boundary controlled port-Hamiltonian systems. For details, see \citep{Lamolinecdc}.
\begin{proposition}
The expected energy increment with respect to the Hamiltonian (\ref{Hamiltonian}) due to the noise effect is given by 
\begin{equation}
\E[dE(\varepsilon(t))| \varepsilon_0=x] - dE(\E[\varepsilon(t)|\varepsilon_0=x]) = \frac{1}{2}\Tr[\mathcal{H}HQH^*] dt, 
\label{energy_incre} 
\end{equation}
where $\varepsilon(t)$ is the stochastic port-Hamiltonian process defined by (\ref{solution_formula}) with $u=0$ and $\mathbb{K}=\mathbb{R}$ and starting at $x\in \mathcal{X}$.
\end{proposition}
\begin{proof}
First, we compute the expected value of the energy of the process $\varepsilon(t)$ starting at $x$. Applying It\^o's formula (\ref{itoformula}), we have
\begin{equation}
\begin{split}
\E\left[dE(\varepsilon(t))| \varepsilon_0=x\right] = \E^x \left[ \psh{E^\prime_x(\varepsilon(t))}{A\varepsilon(t)}_{L^2} dt + \psh{E^\prime_x(\varepsilon(t))}{H dw(t)}_{L^2} +\frac{1}{2} \Tr\left[E^{\prime\prime}_{xx}(\varepsilon(t)) HQH^* \right]dt\right].
\end{split}
\end{equation}
Since $E^\prime_x(\varepsilon(t))$ $= \mathcal{H}\varepsilon(t)$ and $E^{\prime\prime}_{xx}(\varepsilon(t))= \mathcal{H}$ and since the expected value of the increments of the Wiener process vanishes, we get that
\begin{equation}
\E\left[dE(\varepsilon(t))| \varepsilon_0=x\right] =  \E^x\psh{\mathcal{H}\varepsilon(t)}{A\varepsilon(t)}_{L^2} dt + \frac{1}{2} \Tr\left[\mathcal{H} HQH^* \right]dt = \E^x[f_\partial(t)^T e_\partial(t)] dt+ \frac{1}{2} \Tr\left[\mathcal{H} HQH^* \right]dt,
\label{variationineq1}
\end{equation}
where $e_\partial$ and $f_\partial$ are given by (\ref{effort}) and (\ref{flow}), respectively. Second, we compute the expected value of the energy $E$ at time $t$ without noise, which gives
\begin{equation}
dE(x(t)) = f_\partial(t)^T e_\partial(t) dt.
\label{variationineq2}
\end{equation}
Finally, by subtracting (\ref{variationineq2}) from (\ref{variationineq1}) we get (\ref{energy_incre}), i.e. $ \frac{1}{2} \Tr\left[\mathcal{H} HQH^*\right]$ representing the expected energy increment due to the noise effect. 
\end{proof}
\section{Well-posedness}
\label{stochasticwellposed_section}
The notion of well-posedness used here for boundary controlled and observed (BCO) deterministic systems was introduced by Salomon and Weiss, see \cite{staffan,tucsnak,salomon}.
\begin{definition}
The BCO system described by
\begin{align}
\label{pde}
\dot{\varepsilon}(t)&=\mathcal{A}\varepsilon(t), \qquad \varepsilon(0)=\varepsilon_0\in\mathcal{X}\\
\label{inputs}
u(t) &= \mathcal{B}\varepsilon(t),\\
y(t) &= \mathcal{C}\varepsilon(t),
\label{outputs} 
\end{align}
where $\mathcal{A}:D(\mathcal{A})\rightarrow \mathcal{X}$, $\mathcal{B}:D(\mathcal{B})\rightarrow \mathbb{K}^m$ and $\mathcal{C}:D(\mathcal{C})\rightarrow \mathbb{K}^p$ are unbounded linear operators as defined in Definition \ref{def_bcostochastic},
is said to be well-posed if:
\begin{itemize}
\item The operator $A:D(A) \rightarrow \mathcal{X}$ with $D(A)=D(\mathcal{A})\cap ker(\mathcal{B})$ and 
$$A\varepsilon = \mathcal{A}\varepsilon \qquad \text{  for  } \varepsilon\in D(A)$$
is the infinitesimal generator of a $C_0$-semigroup $(T(t))_{t\geq 0}$ on $\mathcal{X}$;
\item There exist $t_f > 0$ and $m_{t_f}\geq 0$ such that the following inequality holds for all $\varepsilon_0\in D(\mathcal{A})$ and $u\in C^2([0, t_f ); \mathbb{K}^m)$ with $u(0) =\mathcal{B}\varepsilon(0)$ (compatibility conditions):
\begin{equation}
\| \varepsilon(t_f)\| _\mathcal{X}^2 + \int_0^{t_f} \| y(t)\| _{\mathbb{K}^p}^2 dt \leq m_{t_f} \left( \| \varepsilon_0\| ^2_\mathcal{X} + \int_0^{t_f} \| u(t)\| _{\mathbb{K}^m}^2 dt \right).
\label{equation_wellposed}  
\end{equation}
\end{itemize}
\label{boundarywell-posed}
\end{definition}
The operators $\mathcal{A}$ and $\mathcal{B}$ are defined by (\ref{A_bc}) and (\ref{inputsto}) respectively. Since the subspaces $D(\mathcal{A})$ and $C^2([0,t_f])$ are dense in $L^2([a,b];\mathbb{K}^n)$ and in $L^2([0,t_f])$ respectively, the inequality (\ref{equation_wellposed}) can be extended to any $\varepsilon_0\in L^2([a,b];\mathbb{K}^n)$ and any $u\in L^2([0,t_f])$. Hence, it entails that for any initial condition in $\mathcal{X}$ and any square integrable input, the mild solution is continuous and the corresponding output is square integrable.\\
Observe that the inequality (\ref{equation_wellposed}) implies that the boundary observation and control operators are admissible for $(T(t))_{t\geq 0}$. We refer the reader to \citep{tucsnak} for further details on admissible observation and control operators. As already pointed in \citep{Lu15}, admissibility is a suitable concept for the study of stochastic well-posed systems.\\ 
We shall now study the effect of randomness on well-posedness, i.e. we shall take into account the stochastic convolution term $\int_0^t T(t-s) H dw(s)$.
\begin{definition}
The BCO stochastic system (\ref{bcostochastic}) is said to be well-posed if:
\begin{itemize}
\item The operator $A:D(A) \rightarrow \mathcal{X}$ with $D(A)=D(\mathcal{A})\cap ker(\mathcal{B})$ and 
$$A\varepsilon = \mathcal{A}\varepsilon \qquad \text{  for  } \varepsilon\in D(A)$$
is the infinitesimal generator of a $C_0$-semigroup $(T(t))_{t\geq 0}$ on $\mathcal{X}$;
\item There exist $t_f > 0$ and $m_f \geq 0$ such that the following inequality holds for all $\varepsilon_0 \in D(\mathcal{A})$ and $u\in C_\mathbb{F}^2([0, t_f ); \mathbb{K}^m)$ with $u(0) =\mathcal{B}\varepsilon(0)$:
\begin{equation}
\begin{split}
\| \varepsilon(t_f)\| ^2_{L_{\mathcal{F}_{t_f}}^2(\Omega;\mathcal{X})} +  \| \mathcal{C}\varepsilon\| ^2_{L^2_{\mathbb{F}}([0,t_f];\mathbb{K}^p)}  \leq & m_{t_f} \left(\| \varepsilon_0\|_{L_{\mathcal{F}_0}^2(\Omega;\mathcal{X})}^2 + \| u(t)\|^2_{L^2_{\mathbb{F}}([0,t_f];\mathbb{K}^m)} + \Tr[Q] \right).
\end{split}
\label{ineqref}  
\end{equation}
\end{itemize}
\label{wellposed_def}
\end{definition}
\begin{rem}
\begin{enumerate}
\item The inequality (\ref{ineqref}) should be interpreted as
\begin{equation}
\E \| \varepsilon(t_f)\| ^2_\mathcal{X} + \E \int_0^{t_f}  \| \mathcal{C}\varepsilon(t)\| ^2_{\mathbb{K}^p} dt \leq m_{t_f} \left( \E\| \varepsilon_0\|^2_{\mathcal{X}} + \E \int_0^{t_f} \| u(t)\| _{\mathbb{K}^m}^2 dt + \Tr[Q] \right).
\label{ineqref2}  
\end{equation}
\item From Theorem \ref{mildsolutionthm}, the process
\begin{equation}
\varepsilon(t) = T(t)\varepsilon_0 + Bu(0) + \int_0^t T(t-s) (\mathcal{A}Bu(s) - B\dot{u}(s)) ds + \int_0^t T(t-s) Hdw(s)+ Bu(t)
\label{mildsolution}
\end{equation}
is the mild solution of the boundary controlled and observed stochastic system (\ref{bcostochastic}) for every $\varepsilon_0\in \mathcal{X}$ and $u\in H^1_\mathbb{F}([0,t_f];\mathbb{K}^m)$. The well-posedness of (\ref{bcostochastic}) entails that the mild solution (\ref{mildsolution}) can be  extended to any $u\in L^2_\mathbb{F}([0,t_f];\mathbb{K}^m)$ s.t. the output process is mean-square integrable. 
\end{enumerate}
\end{rem}
In \cite{zwart, Villegastest} boundary controlled deterministic systems are formulated through the system nodes notation. Based on that, we shall now describe the dynamics of the boundary control system in the stochastic context through this system nodes notation. Consider $t\in[0,T]$ and define the linear operator $S^{b}(t): L^2_{\mathcal{F}_0}(\Omega;\mathcal{X}) \oplus L_\mathbb{F}^2([0,t];\mathbb{K}^m) \oplus M_\mathbb{F}^2([0,t];Z) \rightarrow L^2_{\mathbb{F}}(\Omega;\mathcal{X}) \oplus L^2_{\mathbb{F}}([0,t];\mathbb{K}^p)$ which is given by 
\begin{equation}
 S^{b}(t) \begin{scriptsize}
\left[ \begin{array}{c}
\varepsilon_0\\
u\\
w\\
\end{array} \right]
\end{scriptsize} = \begin{scriptsize}
\left[ \begin{array}{c}
\varepsilon(t)\\
\mathcal{C}\varepsilon(t)
\end{array} \right]
\end{scriptsize}
\end{equation}
on its domain
\begin{equation}
\begin{aligned}
D(S^{b}(t)) = \left\lbrace \begin{scriptsize}
\left[ \begin{array}{c}
\varepsilon_0\\
u\\
w\\
\end{array} \right]
\end{scriptsize}
 \in L^2_{\mathcal{F}_0}(\Omega; \mathcal{X}) \oplus L_\mathbb{F}^2([0,t];\mathbb{K}^m) \oplus M_\mathbb{F}^2([0,t];Z): \varepsilon_0\in D(\mathcal{A}), u\in C_{\mathbb{F}}^2([0,t];\mathbb{K}^m), \mathcal{B}\varepsilon_0=u(0) \right\rbrace
\end{aligned}
\end{equation}
Then we can identify the system's noise in the definition of the operator $S^b(t)$, which yields
\begin{equation}
S^{b}(t) \begin{scriptsize}\left[ \begin{array}{c}
\varepsilon_0\\
u\\
w\\
\end{array} \right]\end{scriptsize}= S(t) \begin{scriptsize}
\left[ \begin{array}{c}
\varepsilon_0\\
u
\end{array} \right]
\end{scriptsize} + \begin{scriptsize}
\left[ \begin{array}{c}
S_1^{w}\\
S_2^{w}
\end{array} \right]
\end{scriptsize}(t) [w],
\end{equation}
and
\begin{equation}
D(S^{b}(t)) = D(S(t)) \oplus D(S^{w}(t)).
\end{equation}
If we assume that the boundary control system (\ref{pde})-(\ref{outputs}) is well-posed according to Definition \ref{boundarywell-posed}, then there exist $t_f>0$ and $m_{t_f}\geq 0$ such that
\begin{equation}
\|S(t_f) \begin{scriptsize}
\left[ \begin{array}{c}
\varepsilon_0\\
u
\end{array} \right]
\end{scriptsize}\|_{\mathcal{X} \oplus L^2([0,t_f];\mathbb{K}^m)}^2
\leq m_{t_f} \| \begin{scriptsize}
\left[
\begin{array}{c}
\varepsilon_0\\
u
\end{array}
\right]
\end{scriptsize}\|_{\mathcal{X}\oplus L^2([0,t_f];\mathbb{K}^m)}.
\end{equation}
In other words, this means that the operator $S(t_f)$ can be extended to a bounded operator $\overline{S}(t_f)$ with $m_{t_f}=\| S(t_f)\| $ since the domain $D(S(t_f))$ is densely defined in $\mathcal{X} \oplus L^2([0,t_f];\mathbb{K}^m)$. Similarly, the operator $S^{w}$ can be extended to a bounded operator since $D(S^w(t))= M^2_{\mathbb{F}}([0,t];Z)$.\\

We shall investigate the well-posedness of BCO stochastic systems in two steps. In the first step it will be shown that if (\ref{ineqref2}) holds for some $t_f>0$, then (\ref{ineqref2}) holds for all $t_f>0$. Next, in the second step we shall consider the well-posedness of a SPHS (\ref{phito})-(\ref{outputsto}) with a deterministic input acting on the mean of the process through the boundaries such that $u(t)=\mathcal{B}\E\left[\varepsilon(t)\right]$, where the leitmotiv will be to consider separately the deterministic and the stochastic dynamics.\\ 
\subsection{Stochastic input $u(t)\in L_\mathbb{F}^2([0,t];\mathbb{K}^m)$} This section contains one of the main results of this paper, namely the extension of \citep[Theorem 13.1.7]{zwart} to the stochastic case.  
\begin{thrm}
If the BCO stochastic system (\ref{bcostochastic}) as in Definition \ref{def_bcostochastic} is well-posed, then for all $t_f>0$ there exists a constant $m_{t_f}>0$ such that (\ref{ineqref2}) holds. 
\label{invariantstochasticwellposed}
\end{thrm}

\begin{proof}
We shall prove the inequality by relying on the well-posedness, i.e., there exist $t_0>0$ and $m_{t_{0}}$ such that 
\begin{equation}
\E \| \varepsilon(t_0)\|^2_\mathcal{X} + \E \int_0^{t_0}  \| \mathcal{C}\varepsilon(t)\| ^2_{\mathbb{K}^p} dt \leq m_{t_0} \left( \E \| \varepsilon_0\|^2_{\mathcal{X}} + \E \int_0^{t_0} \| u(t)\| _{\mathbb{K}^m}^2 dt + \Tr[Q] \right).
\label{wellposedrelation}  
\end{equation}
The main argumentation of the proof is the following: first we shall prove the inequality
\begin{equation}
\E \| \varepsilon(t)\|^2_\mathcal{X} + \E \int_0^{t}  \| \mathcal{C}\varepsilon(s)\|^2_{\mathbb{K}^p} ds \leq m_{t} \left( \E \| \varepsilon_0\|^2_{\mathcal{X}} + \E \int_0^{t} \| u(s)\| _{\mathbb{K}^m}^2 ds + \Tr[Q] \right).
\label{ineqproof}  
\end{equation}
for any $t\in [0,t_0]$ by means of the system nodes formalism; next we shall do it for any $t\in [t_0, 2t_0]$; finally the general case $t> 2^n t_0$ for every $n\in\mathbb{N}$ is deduced by induction.\\ 
\textit{Step 1.}\\
Let $t$ be in $[0, t_0]$. The inequality (\ref{ineqproof}) is given through the system nodes notation by 
\begin{align*}
\|  S^{b}(t) \begin{scriptsize}
\left[ \begin{array}{c}
\varepsilon_0\\
u\\
w\\
\end{array} \right]
\end{scriptsize} \|^2 &\leq m_t  \|\begin{scriptsize}
\left[ \begin{array}{c}
\varepsilon_0\\
u\\
w\\
\end{array} \right]
\end{scriptsize}\|^2_{L^2_{\mathcal{F}_0}(\Omega; \mathcal{X})\oplus L_{\mathbb{F}}^2([0,t];\mathbb{K}^m)\oplus M_\mathbb{F}^2([0,t];Z)} \\
&= m_t \left[ \E \| \varepsilon_0\|_\mathcal{X}^2 + \E \int_0^t \| u(s)\|_{\mathbb{K}^m}^2ds + \Tr\left[Q\right] \right]
\end{align*}
for all $\begin{scriptsize}
\left[ \begin{array}{c}
\varepsilon_0\\
u\\
w\\
\end{array} \right]
\end{scriptsize}$ in the domain
\begin{align*} 
D(S^{b}(t)) = \left\lbrace \begin{scriptsize}
\left[ \begin{array}{c}
\varepsilon_0\\
u\\
w\\
\end{array} \right]
\end{scriptsize}
 \in L^2_{\mathcal{F}_0}(\Omega; \mathcal{X}) \oplus L_\mathbb{F}^2([0,t];\mathbb{K}^m) \oplus M_\mathbb{F}^2([0,t];Z): \varepsilon_0\in D(\mathcal{A}), u\in C_{\mathbb{F}}^2([0,t];\mathbb{K}^m),  \mathcal{B}\varepsilon_0=u(0) \right\rbrace
\end{align*}
The case where $w(t)=0$ is a straightforward adaptation of the argumentation of the deterministic proof with a random variable $\varepsilon_0$ and an $\mathbb{F}$-adapted input $u(t)$. We may take $\varepsilon_0=0$ and $u=0$ hereinafter. 
%
Using the concatenation operator $\diamond$, which is defined for any $L^2$-functions $f,g$ as
\begin{equation}
(f\underset{\tau}{\diamond} g)(t) = \left\lbrace\begin{array}{c}
f(t), \qquad t<\tau,\\
g(t-\tau), t>\tau,
\end{array} \right.
\end{equation}
one observes that $S_1^w(t) [w]$ is bounded for $t\in [0,t_0]$. Indeed,  
\begin{align*}
\|S_1^w(t) [w]\|^2_{L_\mathbb{F}^2(\Omega;\mathcal{X})} &= \|S_1^w(t_0) [0 \underset{t_0-t}{\diamond} w]\|^2_{L_\mathbb{F}^2(\Omega;\mathcal{X})}\\
&\leq m(t_0) \|0 \underset{t_0-t}{\diamond} w\|^2_{M^2_\mathbb{F}([0,t_0];Z)}\\
&= m(t_0) \| w(\cdot-t_0+t)\|^2_{M^2_\mathbb{F}([t_0-t,t_0];Z)}\\
&=m(t_0) \|w(\cdot)\|^2_{M^2_\mathbb{F}([0,t];Z)},
\end{align*}
thanks to the well-posedness at $t_0$ and since a Wiener process is invariant under time translation.\\
Consider the continuous extension $w_{\text{ext}}$ on $[0,t_0]$ of $w(t)$, such that $\mathbb{P}(w_{\text{ext}} = w, \text{  } \forall s\in [0,t])=1$.\\
Since $S^{w}_2(t) \begin{scriptsize}
\left[ w \right]
\end{scriptsize}$ and $S^{w}_2(t_0)\begin{scriptsize}
 \left[ w\right]
\end{scriptsize}$ take values in $L_\mathbb{F}^2([0,t];\mathbb{K}^p)$, we have that 
\begin{equation}
(S^{w}_2(t)w)(s) = (S^{w}_2(t_0)w)(s)
\end{equation}   
for any $s\in[0,t]$. Now consider the particular extension $w_\text{ext}= w \underset{t}{\diamond} 0$. Observe that
\begin{align*}
\E \int_0^t \|  (S^{b}_2(t)w) (s) \| ^2_{\mathbb{K}^p} ds &= \E \int_0^t \| (S^{w}_2(t_0)
w \underset{t}{\diamond} 0)(s)\| ^2_{\mathbb{K}^p} ds\\
&\leq \E \int_0^{t_0} \| (S^{w}_2(t_0) 
w \underset{t}{\diamond} 0)(s)\| ^2_{\mathbb{K}^p} ds\\
& \leq m_{t_0} \|w \underset{t}{\diamond} 0\|^2_{M^2_\mathbb{F}([0,t_0];Z)}\\
&= m_{t_0} \|w \|^2_{M_\mathbb{F}^2([0,t];Z)} = m_{t_0} t\Tr[Q]
\end{align*}
from the well-posedness at $t_0$.\\
\textit{Step 2.}\\
In this third step we prove that the inequality holds for any $t\in[t_0,2t_0]$. Let us consider $t\in [t_0, 2t_0]$ which can be formulated as $t=t_0+t_1$ with $t_1\in [0,t_0]$. Then
\begin{align*}
S^{w}_1(t)w &= \int_{t_0}^t T(t-s) H dw(s) + \int_0^{t_0} T(t-s)H dw(s)\\
&= \int_0^{t_1} T(t_1-r) H d[w(r+t_0)-w(t_0)] + T(t_1) S_1^{w}(t_0)w\\
&= S^{w}_1(t_1)w(t_0+\cdot) + T(t_1) S_1^{w}(t_0)w.
\end{align*}
and  
\begin{align*}
S^{w}_2(t)w(s)= \mathcal{C} \int_0^s T(s-r) Hdw(r) 
=\left\lbrace \begin{array}{ll}
(S_2^w(t_0)w)(s),& s\leq t_0,\\
(S_2^w(t_1)w(t_0+\cdot))(s),& s\in (t_0,t].
\end{array}
\right.
\end{align*}
From Step 1 and Step 2, we deduce that $S^{w}_1(t_1)$ and $S^{w}_2(t_1)$ have bounded extensions and so do $S^{w}_1(t)$ and $S^{w}_2(t)$. Hence, by induction, we can state that the general case $t> 2^n t_0$ holds, which completes the proof. 
\end{proof}
The mild solution for stochastic well-posed systems extends the mild solution as defined in (\ref{mildsolution}) for $u\in L^2_\mathbb{F}([0,t];\mathbb{R}^m)$. The stochastic well-posedness allows us to extend $S^b(t)$ to a bounded linear mapping from $L^2_{\mathcal{F}_0}(\Omega;\mathcal{X}) \oplus L^2_\mathbb{F}([0,t];\mathbb{K}^m) \oplus M^2_\mathbb{F}([0,t];Z)$ to $L^2_{\mathcal{F}_t}(\Omega;\mathcal{X}) \oplus L^2_\mathbb{F}([0,t];\mathbb{K}^p)$.
\subsection{Deterministic input}
The separation of the deterministic and the stochastic dynamics for the sample path $\varepsilon(t)$ and its corresponding output can be done in the following ways, respectively. Let us consider $t\in[0,T]$. From \citep[Corollary 10.1.4]{zwart}, the sample paths $\varepsilon(t)$ given by (\ref{mildsolution}) satisfy the following relation:
\begin{align}
\E \| \varepsilon(s)\|^2_\mathcal{X} &= \E\| T(s) \varepsilon_0+ \int_0^s T(s-r) \mathcal{A}Bu(r)dr - A\int_0^s T(s-r)Bu(r)dr+ \int_0^s T(s-r) Hdw(r) \|^2_\mathcal{X}\nonumber\\
& \leq 3 \E  \| T(s) \varepsilon_0\|^2_\mathcal{X} + 3\|\int_0^s T(s-r) \mathcal{A}Bu(r)dr - A\int_0^s T(s-r)Bu(r)dr\|^2_\mathcal{X}\nonumber\\ 
&+ 3\E \|\int_0^s T(s-r) Hdw(r)\|^2_\mathcal{X}.\label{well-posedineqstate}
\end{align}
For the corresponding output $\mathcal{C}\varepsilon(t)$, we have
\begin{align}
\E \int_0^t \| \mathcal{C}\varepsilon(s)\|^2_{\mathbb{K}^p} ds &= \E \int_0^t \| \mathcal{C}T(s) \varepsilon_0 ds + \mathcal{C}(\int_0^s T(s-r) \mathcal{A}Bu(r) dr - A\int_0^s T(s-r)Bu(r)dr)  \nonumber\\
&\mathcal{C}\int_0^s T(s-r) Hdw(r)\| ^2_{\mathbb{K}^p} ds\nonumber\\
&\leq 3 \E \int_0^t \| \mathcal{C}T(s) \varepsilon_0\|_{\mathbb{K}^p}^2 ds + 3 \E\int_0^t\|\mathcal{C}\int_0^s T(s-r) Hdw(r)\| ^2_{\mathbb{K}^p} ds\nonumber\\ 
&+ 3 \int_0^t \| \mathcal{C}(\int_0^s  T(s-r) \mathcal{A}Bu(r)dr -  A\int_0^s T(s-r)Bu(r)dr)\|^2_{\mathbb{K}^p} ds \label{well-posedineq}
\end{align}



The well-posedness of deterministic port-Hamiltonian systems can be easily verified by checking whether the operator $A$ corresponding to the homogeneous case generates a $C_0$-semigroup, see \citep[Theorem 2.4]{wellposed}. The separation of the dynamics allows us to generalize this result to the stochastic case.

\begin{thrm}[Well-posedness of SPHSs]
Consider the stochastic port-Hamiltonian system
(\ref{phito})-(\ref{outputsto}) satisfying Assumptions \ref{assumptions3} and Condition 4 of Definition \ref{def_bcostochastic}. In addition, assume that:
\begin{enumerate}
\item the multiplication operator $P_1\mathcal{H}$ can be written as 
\begin{equation}
P_1\mathcal{H}(\zeta) = S^{-1}(\zeta) \Delta(\zeta) S(\zeta),\qquad \zeta\in [a,b],
\end{equation}
where $\Delta$ is a diagonal matrix-valued function, $S$ is a matrix-valued function and both $\Delta$ and $S$ are continuously differentiable on $[a,b]$;
\item $HQ^{1/2} Z \subset D(A)$; \label{condtion1}
\item $\int_0^t \| AT(s) H\|^2_{L_2^0} ds <\infty$ for all $t \geq 0$; \label{condtion2}
\end{enumerate}
Then the SPHS (\ref{phito})-(\ref{outputsto}) is well-posed and furthermore, for all $t_f>0$ there exists a constant $m_{t_f}>0$ such that 
\begin{equation}
\E \| \varepsilon(t_f)\| ^2_\mathcal{X} + \E \int_0^{t_f}  \| \mathcal{C}\varepsilon(t)\| ^2_{\mathbb{K}^p} dt \leq m_{t_f} \left( \E\| \varepsilon_0\|^2_{\mathcal{X}} + \int_0^{t_f} \| u(t)\| _{\mathbb{K}^m}^2 dt + \Tr[Q] \right). 
\label{stochasticwellposednineqde} 
\end{equation}
\label{stochasticwell-posedtheorem} 
\end{thrm}
\begin{proof}
In order to separate the deterministic and the stochastic dynamics, we use the inequalities (\ref{well-posedineqstate}) and (\ref{well-posedineq}) to obtain
\begin{align*}
&\E \|\varepsilon(t_f)\|^2_\mathcal{X} +\E \int_0^{t_f} \|\mathcal{C}\varepsilon(s)\|^2_{\mathbb{K}^p} ds\\ 
&\leq 3\E  \| T(t_f) \varepsilon_0\|^2_\mathcal{X} + 3 \|\int_0^{t_f} T(t_f-s) \mathcal{A}Bu(s)ds - A\int_0^{t_f}T(t_f-s)Bu(s) ds\|^2_\mathcal{X}\\ 
&+3\int_0^{t_f} \E \| \mathcal{C}T(s) \varepsilon_0\|^2_{\mathbb{K}^p} ds +3\int_0^{t_f} \|\mathcal{C} \int_0^s  T(s-r) \mathcal{A}Bu(r)dr -\mathcal{C} A\int_0^{s} T(s-r)Bu(r) dr\| ^2_{\mathbb{K}^p} ds\\ 
&+3\E \|\int_0^{t_f} T(t_f-s) Hdw(s)\|^2_\mathcal{X} + 3 \E\int_0^{t_f}\| \mathcal{C}\int_0^s T(s-r) Hdw(r)\|^2_{\mathbb{K}^p} ds.
\end{align*}
The deterministic part had already been set out in \citep[Theorem 2.4]{wellposed}  with $\varepsilon_0=0$. The stochastic part is set out by the admissibility of $\mathcal{C}$, the fact that $H\in L^0_2$ and the following calculation:
\begin{align*}
&\E\int_0^{t_f}\| \mathcal{C}\int_0^s T(s-r) Hdw(r)\| ^2_{\mathbb{K}^p} ds\\
&= \int_0^{t_f} \int_0^s \| \mathcal{C} T(s-r)H Q^{1/2} \|^2_{L_2} dr \text{	}ds\\
&\leq \int_0^{t_f} \int_0^s K\text{	} s \|HQ^{1/2}\|^2_{L_2} dr\text{	} ds + \int_0^{t_f} \int_0^s \|AT(s-r) H Q^{1/2}\|^2_{L_2}  dr\text{	} ds\\
&\leq K(t_f) \|H\|^2_{L_2} \Tr \left[Q \right]
\end{align*}
where we used Assumptions (\ref{condtion1}), (\ref{condtion2}) and the boundedness of $\mathcal{C}$ on the graph norm. This concludes the proof of well-posedness. Moreover, Theorem \ref{invariantstochasticwellposed} entails that the well-posedness holds for any $t_f>0$.
\end{proof}

\section{Illustration on the example of a stochastic vibrating string}
\label{illustration}
In this section we shall focus on a subclass of stochastic port-Hamiltonian systems, namely nice stochastic port-Hamiltonian systems, \citep{francois}. The case of a vibrating string subjected to noise disturbance will be discussed. 
\begin{assumption}
The multiplication operator $P_1^{-1}\mathcal{H}^{-1}$ is assumed to be diagonalizable, i.e.
\begin{equation}
P_1^{-1} \mathcal{H}^{-1}(\zeta) = S(\zeta) A_1(\zeta)  S(\zeta)^{-1},\qquad \zeta \in [a,b],
\end{equation}
where $A_1$ is a diagonal matrix-valued function whose diagonal entries are the eigenvalues $(r_\nu)_{\nu=1}^n$ of $P_1^{-1}\mathcal{H}^{-1}$, whereas $S$ is a matrix-valued function whose columns are corresponding eigenvectors. $S$ and $A_1$ are continuously differentiable on $[a,b]$.   
\label{diagonalizable} 
\end{assumption}

Observe that $P_1^{-1}\mathcal{H}^{-1}$ may have eigenvalues that are not simple. In that case, we shall consider that $P_1^{-1}\mathcal{H}^{-1}$ has $l$ different eigenvalues such that $l\leq n$.
\begin{assumption}
For $\nu\in \lbrace 1,...,l \rbrace$, let us define $R_\nu(z) := \int_a^{z} r_\nu(\zeta) d\zeta$, where $\left(r_\nu(\zeta)\right)_{\nu=1}^l$ are the $l$ different eigenvalues of $P_1^{-1}$ $\mathcal{H}^{-1}(\zeta)$ and $E_\nu(z,\lambda):= e^{\lambda R_\nu(z)} I_{n_\nu}$, where $n_\nu$ is the multiplicity of $r_\nu(\cdot)$ and $I_{n_\nu}$ denotes the $n_\nu$-dimensional unit matrix such that $\displaystyle \sum_{\nu=1}^l n_\nu=n$. We set $E(z,\lambda)=\text{diag}(E_0(z,\lambda),\hdots, E_l(z,\lambda))$, $z\in [a,b]$. We shall assume that the eigenvalue problem 
\begin{equation}
\begin{aligned}
 P_1 \dfrac{d}{d\zeta}((\mathcal{H} x)(\zeta)) + P_0 ((\mathcal{H} x)(\zeta)) &= \lambda x(\zeta),\\
 {W}_B \left[ \begin{array}{c}
(\mathcal{H}x)(b)\\
(\mathcal{H}x)(a)
\end{array}  \right] &=0,
\end{aligned}
\label{eigenproblem}
\end{equation}
is normal, i.e., for sufficiently large $\lambda$, the asymptotic expansion of the characteristic determinant of (\ref{eigenproblem}) given by 
\begin{equation}
p(\lambda) = \sum_{c\in \mathcal{E}} (b_c + \lbrace o(1)\rbrace_{\infty}) e^{\lambda c}
\label{characteristic_det}
\end{equation}
has non-zero minimum and maximum coefficients, where
\begin{equation}
\mathcal{E} = \left\lbrace \sum_{\nu =1}^l \delta_\nu R_\nu(b) : \delta_\nu \in \lbrace0,1\rbrace \right\rbrace \subset \mathbb{R}
\end{equation}
and $\lbrace o(1)\rbrace_{\infty}$ means that for each $c\in\mathcal{E}$ the remaining part depending on $z \in [a,b]$ divided by $\lambda$ tends to $0$ in the uniform norm when $|\lambda|\rightarrow \infty$. 
\label{normal}    
\end{assumption}

The definition of a Riesz-spectral system can be found in \cite{francois, curtain}. The following concept, introduced in \citep{francois}, will turn out to be useful here. 
\begin{definition}
A nice port-Hamiltonian system is a port-Hamiltonian system (according to Definition \ref{sph} with $H=0$) which satisfies Assumptions \ref{diagonalizable} and \ref{normal} and the condition $W_B \Sigma W_B^*\geq 0$ and, whose generator $A$ given by (\ref{operator})-(\ref{domain}) has a uniform gap of eigenvalues, i.e., $\displaystyle \inf_{m \neq p} |\lambda_m-\lambda_p|>0$.
\label{nicephs}
\end{definition}

\begin{thrm}
A Nice port-Hamiltonian system satisfying the assumptions of Definition \ref{nicephs} is a Riesz-spectral system. 
\label{mainresult}
\end{thrm}
Regarding the proof of Theorem \ref{mainresult} and details on the Riesz-spectral property of the port-Hamiltonian framework, see \citep{francois}.
\begin{rem}
\begin{enumerate}
\item The Riesz-basis property allows to derive explicit formulae for the $C_0$-semigroup and the resolvent operator as series of eigenvectors. Furthermore, easily checkable criteria can be undertaken to verify controlability, stability and their dual concepts (observability, detectability).
\item The study of the Riesz-spectral property of nice port-Hamiltonian systems realized in \citep{francois} is based on Tetter's result \citep{Tretter}, which requires some strong assumptions on the eigenvalues of $A$ given by (\ref{operator}) and (\ref{domain}) such as the uniform gap or to have simple eigenvalues. For instance, coupled vibrating strings may have Jordan blocks or a two-dimensional vibrating string will not have a uniform gap of eigenvalues for $A$.  
\end{enumerate}
\end{rem}
\begin{proposition}
For a nice port-Hamiltonian system, the stochastic convolution is given by 
\begin{equation}
W_A(t) =  Hw(t)+ \sum_{k=1}^\infty\sum_{i=1}^\infty\lambda_i e^{\lambda_i t} \int_0^t e^{-\lambda_is} \beta_i(s) ds \psh{H f_i}{\psi_k} \phi_k,
\end{equation}
where $(\beta_i(t))_{i\in\mathbb{N}}$ is a sequence of real independent Wiener processes with increments $(q_i)_{i\in\mathbb{N}}$, $A$ is a Riesz-spectral operator given by (\ref{operator}), which has a discrete spectrum consisting of $\sigma_p(A)=\left\lbrace \lambda_k :k \in\mathbb{N} \right\rbrace$ and whose corresponding eigenvectors 
 $(\phi_k)_{k\in \mathbb{N}}$ form a Riesz basis, $(\psi_k)_{k\in \mathbb{N}}$ are the eigenvectors of the adjoint of $A$ such that $\psh{\phi_k}{\psi_l} = \delta_{kl}$, and $(f_i)_{i\in\mathbb{N}}$ is an orthonormal basis in $Z$.      
 \label{convolutionrepresentation}
\end{proposition}
\begin{proof}
Using an orthonormal basis $(f_i)_{i\in\mathbb{N}}$ and It\^o's formula with $F(s,\beta_i(s)) = e^{-\lambda_k s}\beta_i(s)$, namely 
$$e^{-\lambda_k t}\beta_i(t) = \beta_i(0) + \int_0^t e^{-\lambda_ks} d\beta_i(s) - \int_0^t \lambda_k e^{-\lambda_k s} \beta_i(s) ds,$$
we compute the expression of the stochastic convolution $W_A(t)$: 
\begin{align*}
W_A(t) &= \int_0^t T(t-s) Hdw(s)\\
&= \sum_{i=1}^\infty \int_0^t T(t-s) Hf_i d\beta_i(s)\\
&= \sum_{i=1}^\infty \int_0^t \sum_{k=1}^\infty e^{\lambda_k(t-s)} \psh{Hf_i}{\psi_k}_\mathcal{X} \phi_k d\beta_i(s)\\
&= \sum_{k=1}^\infty\sum_{i=1}^\infty \int_0^t e^{\lambda_k(t-s)} d\beta_i(s) \psh{H f_i}{\psi_k}_\mathcal{X} \phi_k\\
&= \sum_{k=1}^\infty\sum_{i=1}^\infty \beta_i(t) \psh{H f_i}{\psi_k}_\mathcal{X} \phi_k +\lambda_k e^{\lambda_k t} \int_0^t e^{-\lambda_ks} \beta_i(s) ds \psh{H f_i}{\psi_k}_\mathcal{X} \phi_k\\
&= Hw(t)+ \sum_{k=1}^\infty\sum_{i=1}^\infty\lambda_k e^{\lambda_k t} \int_0^t e^{-\lambda_ks} \beta_i(s) ds \psh{H f_i}{\psi_k}_\mathcal{X} \phi_k,
\end{align*}
where we used the stochastic Fubini Theorem, the $Q$-Wiener expansion (\ref{expansion_relation}) and the modal representation of the $C_0$-semigroup $(T(t))_{t\geq 0}$. Define $\psi_k=U f_k$ for all $k\in \mathbb{N}$, where $U$ is an invertible bounded linear operator that transforms $(\psi_k)_k\in\mathbb{N}$ into the orthonormal basis $(f_k)_{k\in\mathbb{N}}$. The stochastic Fubini Theorem follows from
\begin{align*}
\sum_{k=1}^\infty |e^{\lambda_k(t-s)}|^2  |\psh{Hf_i}{\psi_k}_\mathcal{X}|^2 \|\phi_k\|_\mathcal{X}^2 = \sum_{k=1}^\infty e^{2\text{Re} \lambda_k (t-s)} |\psh{U^*H f_i}{f_k}_\mathcal{X}|^2 \|\phi_k\|^2_\mathcal{X} <\infty.
\end{align*}
\end{proof}
As proved in \citep{francois}, the vibrating string with an appropriate choice of boundary conditions is an example of a nice port-Hamiltonian system. From now on, we shall consider this particular example to illustrate the theory presented in the previous sections. Let us recall the example of a vibrating string considered in \citep{francois} and subjected to a spatial-dependent white noise disturbance $\eta$.   
\begin{align}
\label{wave_equation}
\dfrac{\partial^2 z}{\partial t^2}(\zeta,t)&=\dfrac{1}{\rho(\zeta)}\dfrac{\partial}{\partial\zeta} \left( T(\zeta) \dfrac{\partial z}{\partial\zeta}(\zeta,t) \right) + \frac{1}{\rho(\zeta)} \eta(\zeta,t),\\
z(\zeta,0) &= z_0(\zeta),\\
T(a) \frac{\partial z}{\partial\zeta}(a,t)&=u(t), \qquad T(b)\frac{\partial z}{\partial\zeta}(b,t)+ \frac{\partial z}{\partial t}(b,t)= 0, \label{boundaryexample}\\
y(t)&= \frac{\partial z}{\partial t}(a,t),
\label{output_ex}
\end{align}
where $z(\zeta, t)$ is the vertical position of the string at position $\zeta\in [a,b]$ and time $t\in [0,\tau]$. $T(\zeta)$ and $\rho(\zeta)$ are respectively the Young's modulus and the mass density at position $\zeta$. In (\ref{boundaryexample}), the input force $u(t)$ is assumed to be deterministic. The measured output $y(t)$ is the velocity at extremity $a$. The stochastic disturbance is assumed to have intensity $\frac{1}{\rho(\zeta)}$, which means that making the string heavier decreases the impact of the stochastic disturbance.\\
First off, the deterministic dynamic fits in the port-Hamiltonian system. Let us consider $\varepsilon_1(\zeta,t) = \rho(\zeta) \frac{\partial z}{\partial t}(\zeta,t)$ (momentum) and $\varepsilon_2(\zeta, t)=\frac{\partial z}{\partial \zeta}(\zeta,t)$ (strain). Thus, the SPDE (\ref{wave_equation}) can be rewritten as:
\begin{equation}
\frac{\partial}{\partial t}
\left[
 \begin{array}{c}
\varepsilon_1(\zeta,t) \\ 
\varepsilon_2(\zeta,t)
\end{array}
\right] 
= \left[
 \begin{array}{cc}
0 & 1 \\ 
1 & 0
\end{array}
\right] \dfrac{\partial}{\partial \zeta}
\left(
\left[
 \begin{array}{cc}
\frac{1}{\rho(\zeta)} & 0 \\ 
0 & T(\zeta)
\end{array}
\right]
\left[
 \begin{array}{c}
\varepsilon_1(\zeta,t) \\ 
\varepsilon_2(\zeta,t)
\end{array}
\right] \right)+ \left[\begin{array}{c}
1\\
0
\end{array}\right] \eta(\zeta,t)
\label{spde_exampleok}
\end{equation}
where \begin{scriptsize}
$P_1 = \left[
 \begin{array}{cc}
0 & 1 \\ 
1 & 0
\end{array}
\right] $
\end{scriptsize} and
\begin{scriptsize}
$\mathcal{H}(\zeta) = \left[
 \begin{array}{cc}
\frac{1}{\rho(\zeta)} & 0 \\ 
0 & T(\zeta)
\end{array}
\right]$
\end{scriptsize}. The port variables are given by
\begin{align*}
f_\partial(t) = \frac{1}{\sqrt{2}}
\left[
 \begin{array}{c}
T(b) \frac{\partial z}{\partial\zeta}(b,t) - T(a) \frac{\partial z}{\partial\zeta}(a,t)\\ 
\frac{\partial z}{\partial t}(b,t) - \frac{\partial z}{\partial t}(a,t)
\end{array}
\right],\qquad
e_\partial(t) = \frac{1}{\sqrt{2}} \left[
\begin{array}{c}
\frac{\partial z}{\partial t}(b,t) + \frac{\partial z}{\partial t}(a,t)\\
T(b) \frac{\partial z}{\partial \zeta}(b,t) + T(a) \frac{\partial z}{\partial \zeta}(a,t)
\end{array}
\right].
\end{align*}
Thus, the boundary condition becomes in these variables
\begin{equation*}
\left[
\begin{array}{c}
u(t)\\
0
\end{array}
\right]
=
\left[
\begin{array}{c}
T(a) \frac{\partial z}{\partial\zeta}(a,t)\\
T(b) \frac{\partial z}{\partial\zeta}(b,t)+ \frac{\partial z}{\partial t}(b,t)
\end{array}
\right]
=
\frac{1}{\sqrt{2}}
\left[
\begin{array}{cccc}
-1 & 0 & 0 & 1\\
1 & 1 & 1 & 1
\end{array}
\right]
\left[
\begin{array}{c}
f_{\partial}(t)\\
e_\partial(t)
\end{array}
\right]
= W_B 
\left[
\begin{array}{c}
f_{\partial}(t)\\
e_\partial(t)
\end{array}
\right].
\end{equation*}
Similarly, we can rewrite the output equation (\ref{output_ex}) as 
\begin{equation*}
y(t)= \frac{1}{\sqrt{2}}
\left[
\begin{array}{cccc}
0 & -1 & 1 & 0\\
\end{array}
\right]\left[
\begin{array}{c}
f_{\partial}(t)\\
e_\partial(t)
\end{array}
\right]
 = \mathcal{C}\varepsilon(t).
 \end{equation*}
In this case, \begin{small} \text{rank	}
$\left[
\begin{array}{c}
W_{B,1} \\
W_C
\end{array}
\right] = \text{rank	}\left[
\begin{array}{cccc}
-1 & 0 & 0 & 1 \\
0 & -1 & 1 & 0
\end{array}
\right] =2 $ 
\end{small}. We shall now fix an operator $B$ such that $B\in \mathcal{L}(\mathbb{R},\mathcal{X})$, $Bu\in D(\mathcal{A})$, $\mathcal{A}B\in \pazocal{L}(\mathbb{R},\mathcal{X})$ and $\mathcal{B}Bu=u$, i.e.
\begin{equation*}
\mathcal{B}\E B\left( T(a) \varepsilon_2(a,t) \right) = T(a) \varepsilon_2(a,t),
\end{equation*}
Then, \begin{scriptsize}
$Bu = \left[
 \begin{array}{c}
 0\\
 \frac{u}{T(a)}
\end{array}
\right]$
\end{scriptsize} for which $Bu\in D(\mathcal{A})$. Observe that $\mathcal{A}Bu(s)-B\dot{u}(s)\in D(A)=D(\mathcal{A})\cap\text{Ker}\mathcal{B}$ entails that $u(t)=0$, which underscores the need of using the extended form and the Yosida approximation in order to have a strong solution, see Theorem \ref{strongsolution}.\\
From Theorem \ref{generation}, the operator $A$ generates the $C_0$-semigroup $(T(t))_{t\geq 0}$ since $W_B$ has rank $2$ and satisfies $W_B \Sigma W_B^*\geq 0$. We are now in position to verify the assumptions of Theorem \ref{stochasticwell-posedtheorem}. Using a similar calculation as in the proof of Theorem \ref{convolutionrepresentation},
\begin{align*}
\int_0^\tau\| AT(s)H\|^2_{L_2^0} ds
&= \int_0^\tau \Tr\left[A T(s)HQ^{1/2} (A T(s)HQ^{1/2})^* \right] ds\\
&= \sum_{i=1}^\infty \int_0^\tau \psh{AT(s)HQ^{1/2}f_i}{AT(s)HQ^{1/2}f_i}_{\mathcal{X}}ds\\
&= \sum_{i=1}^\infty \int_0^\tau \| A T(s)HQ^{1/2} f_i\| _{\mathcal{X}}^2\\
&= \sum_{i=1}^\infty \int_0^\tau \| \sum_{k=1}^\infty \lambda_k \psh{T(s)HQ^{1/2}f_i}{\psi_k}_{\mathcal{X}} \phi_k\|^2_\mathcal{X} ds\\
&= \sum_{i=1}^\infty \int_0^\tau \| \sum_{k=1}^\infty \lambda_k e^{\lambda_k(s)} \psh{HQ^{1/2}f_i}{\psi_k}_{\mathcal{X}} \phi_k\|^2_\mathcal{X} ds
\end{align*}
where we used the eigenfunction expansion of $A$ and $(T(t))_{t\geq 0}$ with $(\phi_k)_{k\in\mathbb{N}}$ denoting eigenvectors sequences of $A$ that form a Riesz basis and $(\psi_k)_{k\in\mathbb{N}}$ denoting the eigenvector sequences of $A^*$ such that $\psh{\phi_k}{\psi_l}=\delta_{kl}$. This leads to
\begin{align*}
\int_0^\tau\| AT(s)H\|^2_{L_2^0} ds &\leq M \sum_{i=1}^\infty \int_0^\tau  \sum_{k=1}^\infty |\lambda_k|^2 |e^{\lambda_k(s)}|^2 |\psh{HQ^{1/2}f_i}{\psi_k}_{\mathcal{X}}|^2 ds\\
&= M \sum_{i=1}^\infty \sum_{k=1}^\infty |\lambda_k|^2 \int_0^\tau e^{2\text{	Re	} \lambda_k(s)} ds |\psh{HQ^{1/2}f_i}{\psi_k}_{\mathcal{X}}|^2\\
&= \frac{M}{2} \sum_{i=1}^\infty \sum_{k=1}^\infty  \frac{|\lambda_k|^2}{\text{Re	}\lambda_k} (e^{2 \text{	Re	} \lambda_k \tau}-1) q_i |\psh{Hf_i}{\psi_k}_{\mathcal{X}}|^2\\
&\leq K \sum_{i=1}^\infty \sum_{k=1}^\infty (\text{	Im	}\lambda_k)^2 q_i |\psh{Hf_i}{\psi_k}_{\mathcal{X}}|^2 <\infty  
\end{align*}
where the Dominated Convergence Theorem is satisfied under the assumptions that 
\begin{equation}
\left\lbrace
\begin{array}{lr}
\displaystyle  \sum_{i=1}^\infty \sum_{k=1}^\infty (2k+1)^2 \pi^2 (q_i) |\psh{Hf_i}{\psi_k}_{\mathcal{X}}|^2<\infty, & \text{	if	} \sqrt{T\rho} <1\\
\displaystyle  \sum_{i=1}^\infty \sum_{k=1}^\infty (2k)^2 \pi^2 (q_i) |\psh{Hf_i}{\psi_k}_{\mathcal{X}}|^2<\infty, & \text{	if	} \sqrt{T\rho} >1\\
\end{array}\right.,
\end{equation}
where we used the expression of the eigenvalues given in \citep[Section 5]{francois}. Observe that the further assumptions are only made on the noise variance.\\
Moreover, the multiplication operator $P_1\mathcal{H}$ can be rewritten as 
\begin{equation}
P_1\mathcal{H}=
\left[\begin{array}{cc}
\gamma & -\gamma \\ 
\frac{1}{\rho} & \frac{1}{\rho}
\end{array} \right]\left[
\begin{array}{cc}
\gamma & 0 \\ 
0 & -\gamma
\end{array} \right]\left[
\begin{array}{cc}
\frac{1}{2\gamma} & \frac{\rho}{2}\\
-\frac{1}{2\gamma} & \frac{\rho}{2}
\end{array}\right],
\end{equation}
where $\gamma = \sqrt{\frac{T}{\rho}}$. Therefore, the vibrating string described by (\ref{wave_equation})-(\ref{output_ex}) is well-posed with respect to Definition \ref{wellposed_def} and thus, for all $\tau >0$ there exists a constant $m_{\tau}>0$ such that for any $\varepsilon_0\in L^2_{\mathcal{F}_0}(\Omega;\mathcal{X})$ and $u\in L^2([0,\tau];\mathbb{R})$ the inequality (\ref{stochasticwellposednineqde}) holds. 

\section{Conclusion and perspectives} 
In this paper the port-Hamiltonian framework was extended in a stochastic context and some properties of this class were studied such as the existence, the uniqueness and the regularity of the state trajectory and the well-posedness. The aim of studying stochastic port-Hamiltonian systems is to derive a mathematical model for a large class of complex dynamical systems involving boundary control and observation together and possible disturbances on the system.\\
The proposed stochastic port-Hamiltonian framework allows us to prove the existence and uniqueness of weak and strong solutions with a similar approach as in \citep[Chapter 3]{daprato}. Due to the contractivity of the generated $C_0$-semigroup, the mild solution is continuous (see Theorem \ref{continuity}), while mild solutions of SDEs are in most cases only mean-square continuous.\\   
In Section \ref{stochasticwellposed_section} the stochastic counterpart of well-posedness in the sense of Weiss and Salamon was defined with the corresponding system nodes notation. In this study of the proposed well-posedness we distinguished two cases: when the control applied is stochastic and when it is deterministic. In the first case, we showed that if well-posedness is satisfied at least at one time, it holds for any time (see Theorem \ref{invariantstochasticwellposed}). In the second case, the leitmotiv was to separate the deterministic and the stochastic dynamics to prove that under some assumptions SPHSs are stochastically well-posed, see Theorem \ref{stochasticwell-posedtheorem}. Finally, theoretical results were illustrated on the example of a vibrating string by means of a modal representation via a Riesz basis.\\
This paper lays the foundation for the question of well-posedness of infinite-dimensional stochastic port-Hamiltonian systems with boundary control and observation. Further works would be to consider multiplicative noise and noise in the boundary control and/or observation, which would extend the range of considered disturbances. Moreover, in the deterministic case, well-posedness and regularity of the transfer function are closely related. Further works would be to study the regularity of SPHSs.

\appendix
\section{Infinite-dimensional stochastic integration theory}
\label{appendixA}
Some results from the theory of stochastic integration in Hilbert spaces are collected in this appendix for the convenience of the reader. Let $(\Omega, \mathcal{F},\mathbb{F},\mathbb{P})$ be a complete filtered probability space, wherein $\mathbb{F}:= (\mathcal{F}_t)_{t\geq 0}$ and let us consider the Hilbert spaces $\mathcal{X}$ and $Z$ with their respective inner products $\psh{\cdot}{\cdot}_{\mathcal{X}}$ and $\psh{\cdot}{\cdot}_Z$. This appendix is mainly based on \cite{daprato, chow2014stochastic}. 
\begin{definition}
\label{wienerprocessdefi}
A $Z$-valued stochastic process $(w(t))_{t\geq 0}$ is a Wiener process if it satisfies the following conditions:
\begin{enumerate}
\item $w(0)=0$ almost surely;
\item The trajectories $w(t)$ with $t\geq 0$ are continuous;
\item $(w(t))_{t\geq 0}$ has independent increments;
\item $w(t)-w(s)\sim \mathcal{N}(0,(t-s)Q)$ for $t,s\geq 0$.
\end{enumerate}
The covariance operator $Q$ represents the increments of $w(t)$. It is a nonnegative trace class operator and characterizes the distribution of $w(t)$ utterly.   
\end{definition}
In analogy to the Karhunen-Lo\`eve expansion, a Wiener process can be represented as an expansion in the eigenvectors of $Q$, which is given in the following proposition.
\begin{proposition}
If $(w(t))_{t\geq 0}$ is a Wiener process, then there exists a complete orthonormal basis $(v_i)_{i\in\mathbb{N}}$ of $Z$, such that
\begin{equation}
w(t)=\sum_{i=1}^{\infty} \beta_i(t) v_i,
\label{expansion_relation}
\end{equation} 
where $(\beta_i(t))_{i\in\mathbb{N}}$ is a sequence of real independent Wiener processes with increments $(q_i)_{i\in\mathbb{N}}$ such that the series $\sum_{i=1}^\infty q_i$ is convergent.
\end{proposition}

Denote by $Z_0$ the image of the space $Z$ by the square root of the covariance operator: $Z_0:=Q^{1/2}(Z)$, which is a subspace of $Z$ with the norm $\|\cdot\|_0$ and associated with the inner product
$$\psh{u}{v}_0 = \psh{Q^{-1/2}u}{Q^{-1/2 }v}_Z, \qquad u,v\in Z_0,$$
where $Q^{-1/2}$ denotes the pseudo-inverse of $Q^{1/2}$ defined as 
$$(Q^{1/2})^{-1}y:= \argmin \left\lbrace \|z\|_Z: z\in Z, Q^{1/2}z=y \right\rbrace \text{	for all	} y\in \text{	Ran	}Q^{1/2}.$$
Moreover, we consider the space of all Hilbert-Schmidt operators $L_2^0:=L_2(Z_0,\mathcal{X})$, which is a separable Hilbert space equipped with the norm 
\begin{equation}
\|H\|^2_{L_2^0} = \|HQ^{1/2}\|^2_{L_2(Z,\mathcal{X})} = \text{Tr  } \left[ HQ^{1/2} (HQ^{1/2})^* \right] = \text{Tr  } \left[ H Q H^* \right]  
\end{equation} 
for any Hilbert-Schmidt operator $H\in L_2^0$.  We can now turn to the stochastic integral definition of $\mathcal{N}^2_w([0,T];L_2^0)$ integrands, where
$$
\mathcal{N}^2_w([0,T];L_2^0) = \left\lbrace f:[0,T]\to L_2^0 :  \int_0^T ||f(s)||^2_{L_2^0} ds<\infty \right\rbrace.
$$
In order to study port-Hamiltonian systems driven by additive noise, we need to define a stochastic integral of the form 
\begin{equation}
\int_0^T S(t,s) f(s)dw(s),
\label{evolutionequation}
\end{equation}
where $S : [0,T] \times [0,T] \rightarrow \pazocal{L}(\mathcal{X})$ is bounded and strongly continuous for $s,t \in [0,T]$. The special case $S(t,s)= S(t-s)$ is of great importance and is called the convolutional stochastic integral.
\begin{thrm}
Consider a $C_0$-semigroup $(S(t))_{t\geq 0}$ with the infinitesimal generator $A$. If
$\int_0^T \| S(s) f(s) \|_{L_2^0}^2 ds = \int_0^T \Tr\left[ S(s) f(s) Q (S(s)f(s))^* \right] ds$ $< \infty$, then the process $W_A(t):= \int_0^t S(t-s) f(s) dw(s) \in C ([0,T];L^2(\Omega;\mathcal{X}))$ is a Gaussian process with covariance
\begin{equation}
Cov(W_A(T)) = \int_0^T \left[ S(T-s) f(s) Q (S(T-s) f(s))^*\right] ds.
\end{equation}
\label{convolution_stochastic}   
\end{thrm} 
Another important tool that is worth mentioning is It\^o's formula, see \citep[Theorem 4.32]{daprato}.
\begin{thrm}
Let $\phi(s)$ be a $\mathcal{X}$-valued, Bochner integrable mapping on $[0,T]$, $H\in L_2^0$, and let $X_0$ be an $\mathcal{F}_0$-measurable, $\mathcal{X}$-valued random variable. Then,
\begin{equation*}
X(t) := X_0 + \int_0^t \phi(s) ds + \int_0^t H dw(s), \qquad t\in [0,T] 
\end{equation*}
is a well-defined stochastic process. Let $F: [0,T]\times \mathcal{X} \rightarrow \mathbb{K}$ be a continuous function satisfying:
\begin{enumerate}
\item $F(t,x)$ is differentiable in $t$ and $F^\prime_t(t,x)$ is continuous on $[0,T]\times \mathcal{X}$; 
\item $F(t,x)$ is twice Fr\'echet differentiable in $x$, $F^\prime_x(t,x)\in \mathcal{X}$ and $F^{\prime\prime}_{xx}(t,x)\in\pazocal{L}(\mathcal{X})$ are continuous on $[0,T]\times \mathcal{X}$.
\end{enumerate}
Then, $\mathbb{P}$-a.s., for all $t\in[0,T]$
\begin{equation}
\begin{split}
F(t,X(t)) = F(0,X(0)) + \int_0^t \psh{F^\prime_x(s,X(s))}{H dw(s)}_\mathcal{X} + \int_0^t F^\prime_t(s,X(s)) + \psh{F^\prime_x(s,X(s))}{\phi(s)}_\mathcal{X}\\+ \frac{1}{2} \Tr\left[ F^{\prime\prime}_{xx}(s,X(s)) (H Q^{1/2})(H Q^{1/2})^* \right]  ds. 
\end{split}
\label{itoformula}
\end{equation} 
\end{thrm}
\section{Details of Theorem \ref{strongsolution} proof}
\label{appendixB}
To carry out the proof of this result, the following lemma will be helpful. Several results from the Bochner integration have their natural counterparts in stochastic integration. 

\begin{lemma}
\citep[Proposition 4.30]{daprato}\\
If $f(s)Q^{1/2}(Z)\subset D(A)$ and if the following conditions hold: 
\begin{equation}
\int_0^t ||f(s)||^2_{L_2^0} ds <\infty
\end{equation}
and
\begin{equation}
\int_0^t ||Af(s)||^2_{L_2^0} ds <\infty
\end{equation}
then $\int_0^t f(s) dw(s) \in D(A)$ and $A\int_0^t f(s) dw(s) = \int_0^t Af(s) dw(s)$ $\mathbb{P}$-a.s. 
\label{lemmastrongsolution}
\end{lemma}
First, we prove that $X_\lambda^e(t)$ belongs to $D(A^e)$. By assumption, $X^e_0\in D(A^e)$, then $T^e(t)X_0\in D(A)$ by the first part of \citep[Theorem 5.2.2]{zwart}. Moreover, for any $\tilde{u}\in L^2_\mathbb{F}([0,T];\mathbb{K}^m)$, once more from \citep[Theorem 5.2.2]{zwart}, we obtain that $\int_0^t T^e(t-s) B^e \tilde{u}(s)$ $ds \in D(A^e)$. Eventually, by Lemma \ref{lemmastrongsolution}, we find that $\int_0^t T(t-s) H dw(s) \in D(A)$ and that $A\int_0^t T(t-s) Hdw(s) = \int_0^t AT(t-s) Hdw(s)$. This entails that $\int_0^t T^e(t-s)H^edw(s)\in D(A^e)$ and thus concludes the proof for $X^e_\lambda(t)\in D(A^e)$. We are now going to prove that it satisfies (\ref{integralequation}). By assumption, $\int_0^t A^e T^e(t-s) H^e dw(s)$ is well-defined and integrable, then the stochastic Fubini Theorem \citep[Theorem 4.33]{daprato} and \citep[Theorem 5.2.2]{zwart} entail that
\begin{align}
\int_0^t \int_0^s A^eT^e(s-v) H^e dw(v) ds &= \int_0^t \int_v^t A^e T^e(s-v) H^e ds dw(v) \nonumber\\
&= \int_0^t  T^e(t-v) H^e dw(v) - \int_0^t  H^e dw(v).
\label{one} 
\end{align}
Moreover, by stochastic Fubini Theorem once again, 
\begin{align}
\int_0^t \int_0^s A^eT^e(s-v) B_\lambda^e \tilde{u}(v) dv ds &= \int_0^t \int_v^t A^eT^e(s-v) B_\lambda^e \tilde{u}(v) ds dv    \nonumber\\ 
&=\int_0^t  T^e(t-v) B_\lambda^e \tilde{u}(v) dv - \int_0^t B_\lambda^e \tilde{u}(v)  dv.
\label{two}
\end{align}
Since $X^e_\lambda(t)$ is given by
\begin{equation*}
X_\lambda^e(t)= T^e(t)X^e_0 + \int_0^t T^e(t-s) B_\lambda^e \tilde{u}(s) ds + \int_0^t T^e(t-s) H^edw(s),
\end{equation*} 
by applying the operator $A^e$ to both sides and by integrating on $[0,t]$, we get that
\begin{equation}
\begin{aligned}
\int_0^t A^eX_\lambda^e(s) ds = \int_0^t A^eT^e(s) X^e_0 ds + \int_0^t \int_0^s A^e T^e(s-v) H^edw(v) ds
+ \int_0^t \int_0^s A^e T^e(s-v) B_\lambda^e \tilde{u}(v) dv ds. 
\end{aligned}
\end{equation}
Using the relations (\ref{one}) and (\ref{two}), it follows that 
$$X^e_0 + \int_0^t A^e X^e_\lambda(s) ds = X^e_\lambda(t) - \int_0^t H^e dw(s) - \int_0^t B_\lambda^e \tilde{u}(s) ds,$$
which means that $X^e_\lambda(t)$ satisfies the integral equation (\ref{integralequation}).


\bibliographystyle{plain}
\bibliography{biblio_stochastic}  

\end{document}